\documentclass[11pt,twoside,reqno]{amsart}

\usepackage{microtype}
\usepackage{cite}
\usepackage[OT1]{fontenc}
\usepackage{type1cm}
\usepackage{amssymb}
\usepackage{comment}
\usepackage{xcolor}
\usepackage{dsfont}

\usepackage{geometry}
\usepackage{hyperref}
\hypersetup{
  colorlinks=true,
  linkcolor=black,
  anchorcolor=black,
  citecolor=black,
  filecolor=black,      
  menucolor=red,
  runcolor=black,
  urlcolor=black,
}

\numberwithin{equation}{section}

\theoremstyle{plain}
\newtheorem{theorem}{Theorem}[section]

\newtheorem{proposition}[theorem]{Proposition}

\newtheorem{lemma}[theorem]{Lemma}

\theoremstyle{remark}

\newtheorem*{ack}{Acknowledgement}

\theoremstyle{definition}

\newcommand{\LL}{\mathcal{L}}
\newcommand{\PP}{\mathcal{P}}

\newcommand{\R}{\mathbb{R}}

\newcommand{\C}{\mathbb{C}}
\newcommand{\Q}{\mathbb{Q}}

\newcommand{\N}{\mathbb{N}}

\newcommand{\iii}{\mathtt{i}}
\newcommand{\jjj}{\mathtt{j}}

\newcommand{\eps}{\varepsilon}

\newcommand{\dd}{\,\mathrm{d}}

\renewcommand{\ge}{\geqslant}
\renewcommand{\le}{\leqslant}
\renewcommand{\geq}{\geqslant}
\renewcommand{\leq}{\leqslant}

\DeclareMathOperator{\udimloc}{\overline{dim}_{loc}}
\DeclareMathOperator{\ldimloc}{\underline{dim}_{loc}}

\DeclareMathOperator*{\essinf}{ess\,inf}

\DeclareMathOperator{\ldimh}{\underline{dim}_H}

\DeclareMathOperator{\dist}{dist}
\DeclareMathOperator{\diam}{diam}

\DeclareMathOperator{\spt}{spt}

\renewcommand{\atop}[2]{\genfrac{}{}{0pt}{}{#1}{#2}}

\allowdisplaybreaks

\begin{document}

\title{Scaling limits of self-conformal measures}

\author{Bal\'azs B\'ar\'any}
\address[Bal\'azs B\'ar\'any]
        {Department of Stochastics \\
         Institute of Mathematics \\
        Budapest University of Technology and Economics \\
         M\H{u}egyetem rkp 3 \\
         H-1111 Budapest \\
         Hungary}
\email{balubsheep@gmail.com}

\author{Antti K\"aenm\"aki}
\address[Antti K\"aenm\"aki]
        {Research Unit of Mathematical Sciences \\ 
         P.O.\ Box 8000 \\ 
         FI-90014 University of Oulu \\ 
         Finland}
\email{antti.kaenmaki@oulu.fi}

\author{Aleksi Py\"or\"al\"a}
\address[Aleksi Py\"or\"al\"a]
        {Research Unit of Mathematical Sciences \\ 
         P.O.\ Box 8000 \\ 
         FI-90014 University of Oulu \\ 
         Finland}
\email{aleksi.pyorala@oulu.fi}

\author{Meng Wu}
\address[Meng Wu]
        {Research Unit of Mathematical Sciences \\ 
         P.O.\ Box 8000 \\ 
         FI-90014 University of Oulu \\ 
         Finland}
\email{meng.wu@oulu.fi}

\thanks{B.\ B\'ar\'any was supported by the grants NKFI FK134251, K142169, and the grant NKFI KKP144059 ``Fractal geometry and applications''.  M.\ Wu was supported by the Academy of Finland, grant No.\ 318217.}
\date{\today}

\begin{abstract}
  We show that any self-conformal measure $\mu$ on $\R$ is uniformly scaling and generates an ergodic fractal distribution. This generalizes existing results by removing the need for any separation condition. We also obtain applications to the prevalence of normal numbers in self-conformal sets, the resonance between self-conformal measures on the line, and projections of self-affine measures on carpets.
\end{abstract}

\maketitle

\section{Introduction}

Taking tangents and studying an object through its small-scale structure is a classical approach in analysis. Such an idea also finds applications in geometric measure theory: For a Radon measure $\mu$ on $\R^d$, a point $x$ in its support $\spt\mu$, and a number $t\geq 0$, let $\mu_{x,t}$ be the \emph{magnification} of $\mu$ at $x$ by $e^t$, that is, the probability measure on the unit ball given by 
$$
  \mu_{x,t}(A) = \frac{\mu(e^{-t}A + x)}{\mu(B(x,e^{-t}))}
$$
for every Borel set $A\subseteq B(0,1)$. The family $(\mu_{x,t})_{t\geq0}$ is called the \emph{scenery} of $\mu$ at $x$. The statistical properties of the scenery have substantial implications on the local geometry of $\mu$ at $x$ and to capture these statistics, we define the \emph{scenery flow} of $\mu$ at $x$ as the one-parameter family $(\langle \mu \rangle_{x,T})_{T > 0}$, where
\begin{equation}\label{eq-sceneryflow}
  \langle \mu \rangle_{x,T} = \frac{1}{T}\int_0^T \delta_{\mu_{x,t}}\dd t
\end{equation}
and $\delta_\nu$ denotes the Dirac mass at $\nu$. The accumulation points of \eqref{eq-sceneryflow} in the weak$^*$-topology are called \emph{tangent distributions} of $\mu$ at $x$. If there exists a measure $P$ such that for $\mu$-almost every $x$, $P$ is the unique tangent distribution at $x$, we say that $\mu$ is \emph{uniformly scaling} and generates $P$.

The notion of uniform scaling appeared first in a work of Gavish \cite{Gavish2011}, where he also showed that self-similar measures on $\R^d$ with the open set condition are uniformly scaling. Recall that a measure $\mu$ is self-similar if there exist affine, contracting, and angle preserving maps $f_1,\ldots, f_n$ and a probability vector $(p_1,\ldots,p_n)$ such that 
\begin{equation}\label{eq-selfsimilardef}
  \mu = \sum_{i=1}^n p_i f_*\mu.
\end{equation}
The proof given by Gavish relied on the theory of Furstenberg's CP-chains, and a more direct ergodic-theoretic proof was introduced by Hochman \cite{Hochmanpreprint} in his systematic study of tangent distributions. Later, using a technique similar to that of Gavish, with the additional machinery developed in \cite{Hochmanpreprint}, Hochman and Shmerkin \cite{HochmanShmerkin2015} showed that also self-conformal measures on $\R$ with the open set condition are uniformly scaling. Recall that a measure is self-conformal if it satisfies \eqref{eq-selfsimilardef} with $f_i$'s being injective and differentiable on some bounded open and convex set $U\subset \R^d$, with their differentials $x\mapsto D f_i|_x$ being Hölder continuous, $\sup_{x\in U} \Vert D f_i|_x\Vert < 1$, $D f_i|_x\neq 0$, and $\Vert Df_i|_x\Vert^{-1}Df_i|_x$ being orthogonal for every $x\in U$. In this case, we say that $\Phi = (f_i)_{i \in \Lambda}$ is a \emph{conformal iterated function system}. In \cite{Pyorala2022}, the third author showed that for a self-similar measure to be uniformly scaling it suffices to assume only the weak separation condition.

Self-affine measures are measures that satisfy \eqref{eq-selfsimilardef} with $f_i$'s being affine maps. Regarding such measures in higher dimensions, Ferguson, Fraser, and Sahlsten \cite{FergusonFraserSahlsten2015} have shown that self-affine measures on Bedford-McMullen carpets in $\R^2$ are uniformly scaling. This more or less amounts to what is currently known of the scenery flow of self-affine measures. While Kempton \cite{Kempton2015} showed that self-affine measures with an irreducibility condition are not necessarily uniformly scaling, the question of general reducible self-affine measures remains open. 

In this work, we obtain a complete resolution to the question of the scaling properties for all self-conformal measures in $\R^d$, removing the need for any separation conditions from the assumptions of \cite{Gavish2011, HochmanShmerkin2015, Pyorala2022}, and establish uniformly scaling property on $\R$. As a consequence, we obtain applications to the prevalence of normal numbers in self-conformal sets and to the dimensions of convolutions of self-conformal measures on the line, both of which are active research topics on their own. Furthermore, we prove that self-affine measures on any carpet satisfying the strong separation condition are uniformly scaling and, as a consequence, show that the dimension of every non-principal orthogonal projection of the self-affine measure is preserved.

Before formulating the main results, let us introduce some notation. Let $\Lambda$ be a finite set having at least two elements and let $\Phi=(f_i)_{i\in \Lambda}$ be a conformal iterated function system on $\R^d$ with attractor $K$. A measure $\bar{\mu}$ on $\Lambda^\N$ is called \emph{quasi-Bernoulli} if there exists $C \ge 1$ such that 
\begin{equation} \label{eq:quasi-bernoulli-def}
  C^{-1} \bar{\mu}([\mathtt{i}])\bar{\mu}([\mathtt{j}]) \leq \bar{\mu} ([\mathtt{ij}]) \leq C \bar{\mu}([\mathtt{i}])\bar{\mu}([\mathtt{j}])
\end{equation}
for all finite words $\mathtt{i}$ and $\mathtt{j}$ formed from the elements of $\Lambda$. Here $[\mathtt{i}]$ denotes the collection of infinite words with prefix $\mathtt{i}$. If $C \ge 1$ above can be chosen to be $1$, then $\bar{\mu}$ is a \emph{Bernoulli measure}. Let $\Pi$ be the canonical projection $\Lambda^\N\to K$; see \S \ref{sec:symbolic-space} for the definition. If $\bar{\mu}$ is quasi-Bernoulli, then, by slightly abusing notation, we also call the measure $\mu=\Pi_*\bar{\mu}$ on $K$ quasi-Bernoulli. In particular, any self-conformal measure is Bernoulli. For the definition of an ergodic fractal distribution, see \S \ref{sec:tangent-distribution}.

\begin{theorem}\label{thm-unif-scaling-1}
  Let $\Phi$ be a conformal iterated function system on $\R^d$ with attractor $K$ and $\mu$ be a quasi-Bernoulli measure on $K$. If $d=1$ or $\Phi$ consists of similarities, then $\mu$ is uniformly scaling and generates an ergodic fractal distribution.
\end{theorem}

In higher dimensions, self-conformal measures are in general not uniformly scaling. The following counter-example is by Fraser and Pollicott \cite{FraserPollicott2015}: The upper half of the boundary of the unit circle in $\R^2$ is a self-conformal set and the normalized one-dimensional Lebesgue measure $\mu$ on the semi-circle is a self-conformal measure. It is clear that for any point $x$ on the semi-circle all tangent measures of $\mu$ at $x$ are supported on lines parallel to the tangent to the semi-circle at $x$.  Thus for distinct $x$ and $y$ the collections of tangent measures of $\mu$ at $x$ and $y$ do not intersect and, in particular, the measure $\mu$ is not uniformly scaling. Nevertheless, we can state the following.

\begin{theorem}\label{thm-unif-scaling-2}
    Let $\Phi$ be a conformal iterated function system on $\R^d$ with attractor $K$ and $\mu$ be a quasi-Bernoulli measure on $K$. Then there exists an ergodic fractal distribution $P$ such that for $\mu$-almost every $x$ and any tangent distribution $Q$ at $x$ there exists a measure $\zeta$ on the orthogonal group of $\R^d$ such that $Q = \int O_*P\dd \zeta(O)$.
\end{theorem}

Alternatively, we can prove that for some diffeomorphism $h$, the measure $h_*\mu$ is uniformly scaling along a subsequence.

\begin{theorem}\label{thm-unif-scaling-3}
    Let $\Phi$ be a conformal iterated function system on $\R^d$ with attractor $K$ and $\mu$ be a quasi-Bernoulli measure on $K$. Then there exists an ergodic fractal distribution $P$, a sequence $(n_k)_{k\in\N}$, and a function $h\in \mathcal{C}^{1+\alpha}(\R^d)$ such that
    \begin{equation*}
        \lim_{k\to\infty} \frac{1}{n_k} \int_0^{n_k} \delta_{(h_*\mu)_{x,t}}\dd t = P
    \end{equation*}
    for $h_*\mu$-almost all $x$.
\end{theorem}

The proofs of Theorems \ref{thm-unif-scaling-1}--\ref{thm-unif-scaling-3} are postponed until \S \ref{section: proof-scenery-flow}. We remark that the proofs are somewhat indirect: Instead of studying the scenery $(\mu_{x,t})_{t\geq 0}$ directly, we rely on a result of Hochman \cite{Hochmanpreprint} that for any Radon measure $\mu$ typical \emph{tangent measures}, i.e., the weak$^*$-accumulation points of the sequence $(\mu_{x,t})_{t\geq 0}$ are uniformly scaling and generate ergodic fractal distributions. When $\mu$ is self-conformal, its tangent measures have a very special structure: they are of the form
\begin{equation}\label{eq-tangentmeasure}
  \int h_*\mu \dd\nu(h)
\end{equation}
for some measure $\nu$ in the space of diffeomorphisms $\R^d\to\R^d$. This is almost the case for quasi-Bernoulli measures as well, as will be seen in Lemma \ref{lemma-structureoftangents}. Furthermore, another general result of Hochman \cite{Hochmanpreprint} states that many tangent measures have Hausdorff dimension at most that of $\mu$. The proof is then concluded by the key technical result of the paper, Proposition \ref{prop-main-1}, which states that the sceneries of a measure of the form \eqref{eq-tangentmeasure}, with dimension close to that of $\mu$, will be asymptotic to the sceneries of $\mu$ at many points. This will allow us to transfer the regularity of local statistics of typical tangent measures back to the regularity of local statistics of $\mu$.

We will next discuss some applications of Theorem \ref{thm-unif-scaling-1}. For a contracting diffeomorphism $f\colon\R\to\R$ we define the \emph{asymptotic contraction ratio} of $f$ to be $\lambda(f) = Df|_p$, where $p$ is the fixed point of $f$. Recall that for $\beta>1$ a real number $x$ is called $\beta$-\emph{normal} if the sequence $(\beta^n x \mod 1)_{n\in\N}$ equidistributes for the unique absolutely continuous $\times \beta$-invariant measure on the torus. When $\beta\in\N$ such an invariant measure is just the Lebesgue measure. Let us also recall from Hochman and Shmerkin \cite{HochmanShmerkin2015} that an iterated function system $\Phi=(f_i)_{i\in\Lambda}$ is said to be \emph{totally non-linear} if it is not conjugate to a linear iterated function system via $\mathcal{C}^1$ maps, that is, there is no invertible $g\in \mathcal{C}^1(\R)$ such that the iterated function system $g\Phi=(g\circ f_i\circ g^{-1})_{i\in\Lambda}$ consists only of linear maps.

Combining Theorem \ref{thm-unif-scaling-1} with the methods of Hochman and Shmerkin \cite{HochmanShmerkin2015}, we obtain the following result on normal numbers in self-conformal sets. The proof of the result can be found in \S \ref{section-normalnumbers}. We say that a conformal iterated function system $\Phi = (f_i)_{i\in\Lambda}$ on $\R$ is \emph{arithmetically independent} of a real number $\beta>1$ if zero is an accumulation point of the set 
\begin{equation*}
    \lbrace \log \lambda(f_\iii) + n\log\beta: \iii\in\Lambda^* \text{ and } n\in\N\rbrace.
\end{equation*}
For example, if $\frac{\log \lambda(f_i)}{\log \beta}\not\in\Q$ for some $i\in\Lambda$, then $\Phi$ is arithmetically independent of $\beta$. Recall that a \emph{Pisot number} is a real algebraic integer greater than $1$, all of whose other Galois conjugates are less than $1$ in absolute value.

\begin{theorem}\label{pointwise_normality}
Let $\Phi = (f_i)_{i \in \Lambda}$ be a conformal iterated function system on $\R$ with attractor $K$, $\mu$ be a non-atomic quasi-Bernoulli measure on $K$, and $\beta > 1$ be a Pisot number. If $\Phi$ is arithmetically independent of $\beta$, then for any continuously differentiable diffeomorphism $h$, $h_*\mu$-almost every $x$ is $\beta$-normal. Furthermore, if $\Phi$ is totally non-linear and consists of real analytic maps, then the statement holds for every Pisot number $\beta > 1$.
\end{theorem}

We remark that Theorem \ref{pointwise_normality} generalizes the recent result of Algom, Baker, and Shmerkin \cite{AlgomBakerShmerkin2022} into the self-conformal setting. In the case when $\mu$ is self-conformal, $\beta$ is an integer, and $\Phi$ has certain non-linearity, the conclusion of Theorem  \ref{pointwise_normality} can be alternatively obtained by combining the results in \cite{AlgomHertzWang2021,AlgomHertzWang2023,BakerSahlsten2023} on Fourier decay of self-conformal measures with the Davenport-Erd\H{o}s-LeVeque criterion for pointwise normality. We note that the Fourier decay results in \cite{AlgomHertzWang2021,AlgomHertzWang2023,BakerSahlsten2023} rely on $\mu$ being the projection of a Bernoulli measure on $\Lambda^\N$, and do not seem to have applications on $\beta$-normal numbers when $\beta$ is non-integer.

Combining Theorem \ref{thm-unif-scaling-1} with a version of the local entropy averages of Hochman and Shmerkin \cite{HochmanShmerkin2012} found in \cite{Hochmanpreprint}, we obtain the following result on the resonance between quasi-Bernoulli measures on self-conformal sets. We say that two exact-dimensional probability measures $\mu$ and $\nu$ on $\R$ \emph{resonate}, if
$$
\dim (\mu*\nu) < \min\lbrace 1,\dim\mu + \dim\nu\rbrace;
$$
otherwise they are said to \emph{dissonate}. Two conformal iterated function systems $\Psi = (f_i)_{i \in \Gamma}$ and $\Phi = (g_j)_{j \in\Lambda}$ on $\R$ are called \emph{independent} if zero is an accumulation point of the set
$$
\{\log \lambda(f_{\mathtt{i}})-\log \lambda(g_\mathtt{j}) : \mathtt{i}\in \Gamma^* \text{ and } \mathtt{j}\in \Lambda^*\}.
$$
For instance, if $\frac{\log \lambda(f_i)}{\log \lambda(g_j)} \not\in \Q$ for some $(i,j) \in \Gamma \times \Lambda$, then $\Psi $ and $\Phi$ are independent.

\begin{theorem}\label{prop-dissonance}
Let $\Psi = (f_i)_{i \in \Gamma}$ and $\Phi = (g_j)_{j \in\Lambda}$ be conformal iterated function systems on $\R$ with attractors $K_\Psi$ and $K_\Phi$, respectively. If $\Psi $ and $\Phi$ are independent, then for any quasi-Bernoulli measures $\mu$ on $K_\Psi$ and $\nu$ on $K_\Phi$ we have
$$
\dim(\mu*\nu) = \min \lbrace 1, \dim \mu + \dim \nu \rbrace,
$$
that is, $\mu$ and $\nu$ dissonate.
\end{theorem}

The proof of Theorem \ref{prop-dissonance} is postponed until \S \ref{section-dissonance}. The question of resonance between dynamically defined sets and measures has a long history, dating back to Furstenberg's conjecture on sums of $\times 2$ and $\times 3$-invariant sets from the 1960s: In a sense, the general idea is that two dynamically defined objects should not resonate unless the objects or the defining dynamics are arithmetically similar. Analogues of Theorem \ref{prop-dissonance} have been verified for various classes of dynamically defined sets and measures; for example, see \cite{PeresShmerkin2009, Moreira1998, NazarovPeresShmerkin2012, HochmanShmerkin2012}. Very recently, Bruce and Jin \cite{BruceJin2022} proved a version of Theorem \ref{prop-dissonance} for a rich class of dynamically defined measures on homogeneous self-similar sets, containing canonical projections of ergodic measures. We remark that Theorem \ref{prop-dissonance} does not easily extend to ergodic measures, while the results of \cite{BruceJin2022} do not appear to be extendable to measures on self-conformal sets. Previously, Theorem \ref{prop-dissonance} was known for self-conformal measures only under the strong separation condition; see \cite[Theorem 1.39]{Hochmanpreprint}. It is also an interesting question whether an analog of Theorem \ref{prop-dissonance} holds in higher dimensions. Very recently, the third author showed in \cite{Pyorala2023} that this is the case for a class of planar self-affine measures.

Finally, let us assume that the maps $f_i\colon \R^2 \to \R^2$, $f_i(x)=A_ix+a_i$, are affine such that
\begin{equation*}
  A_i =
  \begin{pmatrix}
    \rho_i & 0 \\ 
    0 & \lambda_i
  \end{pmatrix}
\end{equation*}
for all $i \in \Lambda$. In this case, we say that $\Phi = (f_i)_{i \in \Lambda}$ is a \emph{diagonal iterated function system}. The following projection theorem generalizes that of Ferguson, Fraser, and Sahlsten \cite{FergusonFraserSahlsten2015} for measures with no grid structure, i.e., without any assumption on the translation vectors $a_i$. We say that $\Phi$ satisfies the \emph{rectangular strong separation condition} if $f_i([-1,1]^2) \cap f_j([-1,1]^2) = \emptyset$ whenever $i \ne j$. A diagonal iterated function system $\Phi$ has {\em independent eigenvalues} if zero is an accumulation point of the set
$$
\{\log \lambda_{\mathtt{i}}-\log \rho_{\mathtt{j}}: \mathtt{i},\mathtt{j}\in \Lambda^*\}.
$$
For instance, if $\frac{\log \lambda_{i}}{\log \rho_{i}} \not\in \Q$ for some $i\in \Lambda$, then $\Phi$ has independent eigenvalues.

\begin{theorem} \label{self-affine_projection}
  Let $\Phi$ be a diagonal iterated function system on $\R^2$ with attractor $K$ satisfying the rectangular strong separation condition. If $\Phi$ has independent eigenvalues, then for any self-affine measure $\mu$ on $K$ with simple Lyapunov spectrum, we have
  $$
    \dim \pi_*\mu = \min\lbrace 1,\dim\mu\rbrace
  $$
  for all non-principal orthogonal projections $\pi$.
\end{theorem}

The rest of the paper is organized as follows. In \S \ref{section-preliminaries}, we collect some preliminaries and introduce notation. The proofs of Theorems \ref{thm-unif-scaling-1}--\ref{thm-unif-scaling-3} are given in \S \ref{section: proof-scenery-flow}, and the applications to normal numbers and resonance, Theorems \ref{pointwise_normality} and \ref{prop-dissonance}, are proven in \S \ref{section-normalnumbers} and in \S \ref{section-dissonance}, respectively. Finally, in \S \ref{sec-self-affine-proj}, we show that self-affine measures on attractors of diagonal iterated function systems are uniformly scaling and conclude the proof of Theorem \ref{self-affine_projection}.

\section{Preliminaries}\label{section-preliminaries}

For $\alpha \geq 0$ we denote by $\mathcal{C}^{1+\alpha}(\R^d)$ the family of functions $f$ for which there exists a bounded open and convex set $U\subseteq \R^d$ such that $f\colon U \to \R^d$ is differentiable, injective, its differential $x\mapsto Df|_x$ is $\alpha$-Hölder continuous on $U$, $Df|_x \neq 0$ for all $x \in U$, $\Vert Df|_x\Vert^{-1} Df|_x$ is orthogonal for all $x\in U$, and $\sup_{x\in U}\Vert Df|_x\Vert < 1$. Such functions are called \emph{conformal} or, when $d=1$, simply \emph{$\mathcal{C}^{1+\alpha}$-functions}.
 
The collection of all Borel probability measures on a metric space $X$ is denoted by $\mathcal{P}(X)$.  For $\mu,\nu\in \mathcal{P}(X)$ we use the \emph{L\'evy-Prokhorov metric} to measure their distance:
$$
  d_{\rm LP}(\mu,\nu)=\inf \{\varepsilon>0: \mu(A)\le \nu(A^{\varepsilon})+\varepsilon \textrm{ for all Borel sets }A\},
$$
where $A^{\varepsilon}$ is the $\varepsilon$-neighborhood of $A$ in $X$. Note that if $X$ is separable, then the convergence in L\'evy-Prokhorov metric is equivalent to the weak convergence and also the space $(\mathcal{P}(X),d_{\rm LP})$ is separable. Given a measure $\mu$ and a measurable function $f$, we denote the push-forward measure by $f_*\mu = \mu\circ f^{-1}$. We let $\mathcal{D}_n(\R^d)$ be the partition of $\R^d$ to dyadic cubes of side length $2^{-n}$ homothetic to $[0,1)^d$ such that the origin is a vertex. We denote the unique element of $\mathcal{D}_n(\R^d)$ containing $x \in \R^d$ by $\mathcal{D}_n(x)$. For $\mu\in \mathcal{P}(\R^d)$ and $x\in \spt\mu$ we define the \emph{dyadic magnification} of $\mu$ at $x$ by
\begin{equation}\label{eq:muDn}
  \mu^{\mathcal{D}_n(x)}=\frac{1}{\mu(\mathcal{D}_n(x))} (f_{\mathcal{D}_n(x)})_*(\mu|_{\mathcal{D}_n(x)}),
\end{equation}
where $\mu|_{\mathcal{D}_n(x)}$ is the restriction of $\mu$ to $\mathcal{D}_n(x)$ and $f_{\mathcal{D}_n(x)}$ is the unique homothety sending $\mathcal{D}_n(x)$ to $[0,1)^d$. For Radon measures $\mu$ and $\nu$ and $C \geq 1$ we write $\mu \sim_C\nu$ to indicate that $C^{-1} \leq \frac{\mathrm{d}\mu }{\mathrm{d}\nu} \leq C$, where $\frac{\mathrm{d}\mu }{\mathrm{d}\nu}$ is the Radon-Nikodym derivative.

\subsection{Symbolic space} \label{sec:symbolic-space}

Let $\Lambda$ be a finite set with $\#\Lambda\geq 2$. Let $\Lambda^* = \bigcup_{n\in\N} \Lambda^n$ be the set of finite words and for each $\mathtt{i} \in \Lambda^*$ write $|\mathtt{i}|$ to denote the unique integer for which $\mathtt{i}\in\Lambda^{|\mathtt{i}|}$. For a finite word ${\mathtt{i}} = i_1\cdots i_{n-1}i_n \in \Lambda^n$, write $\mathtt{i}^- = i_1\cdots i_{n-1} \in \Lambda^{n-1}$. For an infinite word $\mathtt{j} \in \Lambda^\N$ and $k\in\N$ let $\mathtt{j}|_k \in\Lambda^k$ be the projection of $\jjj$ onto the first $k$ coordinates. Let $\sigma \colon \Lambda^\N \to \Lambda^\N$ be the \emph{left shift} defined by $\sigma\jjj = j_2j_3\cdots$ for all $\jjj = j_1j_2\cdots$. We equip $\Lambda^\N$ with the topology generated by the \emph{cylinder sets}:
$$
  [\mathtt{i}] = \lbrace \mathtt{j}\in\Lambda^\N:\ \mathtt{j}|_{|\mathtt{i}|} = \mathtt{i}\rbrace
$$
for all $\mathtt{i}\in\Lambda^*$. Given an iterated function system $\Phi = (f_i)_{i\in\Lambda}$ with attractor $K$, the \emph{canonical projection} $\Pi \colon \Lambda^\N\to K$ is defined by setting 
$$
\Pi(\mathtt{i}) = \lim_{n\to\infty} f_{i_1}\circ\cdots\circ f_{i_n}(0)
$$
for all $\mathtt{i}=i_1i_2i_3\cdots \in \Lambda^\N$. Note that $\Pi$ is a continuous surjection. To simplify notation we write $f_\mathtt{i} = f_{i_1}\circ\cdots \circ f_{i_n}$ and, if $p_i \in \R$ for all $i \in \Lambda$, also $p_\mathtt{i} = p_{i_1}\cdots p_{i_n}$ for all finite words $\mathtt{i} = i_1\ldots i_n\in\Lambda^*$.

\subsection{Tangent distributions} \label{sec:tangent-distribution}

We recall some of the theory of tangent distributions developed by Hochman in \cite{Hochmanpreprint}. Let $(S_t)_{t \ge 0}$ be the flow in the space $\{\mu \in \PP(\R^d) : 0 \in \spt\mu\}$ defined by
\begin{equation*}
  S_t\mu = \mu_{0,t}.
\end{equation*}
Given $\mu \in \PP(\R^d)$ and $x \in \spt\mu$, the scenery of $\mu$ at $x$ can thus be viewed as the orbit of $(T_x)_*\mu$ under the flow $(S_t)_{t \ge 0}$ as $\mu_{x,t} = S_t(T_x)_*\mu$ for all $t \ge 0$, where $T_x$ is the translation $y \mapsto y-x$.

Let $P$ be a Borel probability measure on $\mathcal{P}(B(0,1))$. We say that $P$ is a \emph{fractal distribution} if it is invariant under $(S_t)_{t\geq0}$ and if for any $\mathcal{A}\subseteq \mathcal{P}(B(0,1))$ with $P(\mathcal{A}) = 1$, for almost every $\mu\in \mathcal{A}$, also $\mu_{x,t}\in \mathcal{A}$ for $\mu$-almost all $x\in B(0,1)$ and every $t\geq 0$ for which $B(x,e^{-t})\subseteq B(0,1)$. Although the flow $(S_t)_{t \ge 0}$ is not continuous and the defining property of the fractal distribution is not closed, the family $\mathcal{FD}$ of all fractal distributions is closed in the weak$^*$-topology; see \cite[Theorem 1.1]{KaenmakiSahlstenShmerkin2015}. Furthermore, we say that $P$ is an \emph{ergodic fractal distribution} if it is a fractal distribution and ergodic under $(S_t)_{t\geq0}$. The set of ergodic fractal distributions is dense in $\mathcal{FD}$; see \cite[Theorem 1.2]{KaenmakiSahlstenShmerkin2015}.

Recall that tangent distributions of $\mu$ at $x$ are the accumulation points of the scenery flow $(\langle \mu \rangle_{x,T})_{T > 0}$ in the weak$^*$-topology. In addition, $\mu$ is uniformly scaling and generates $P$ if $P$ is the unique tangent distribution of $\mu$ at $\mu$-almost all $x$. The following result of Hochman \cite[Theorem 1.6]{Hochmanpreprint} shows that typical measures drawn from fractal distributions are much more structured than arbitrary measures. 

\begin{theorem} \label{thm-Hochman1.6}
  Let $P$ be a fractal distribution. Then $P$-almost every $\nu$ is uniformly scaling and generates the ergodic component of $P$ into which it belongs.
\end{theorem}

In fact, every fractal distribution is generated by some uniformly scaling measure; see \cite[Theorem 1.3]{KaenmakiSahlstenShmerkin2015}. In the other direction, Hochman \cite[Theorem 1.7]{Hochmanpreprint} has shown that tangent distributions of any Radon measure have the structure of a fractal distribution.

\begin{theorem} \label{thm-Hochman1.7}
  Let $\mu$ be a Radon measure. Then for $\mu$-almost every $x$, every tangent distribution of $\mu$ at $x$ is a fractal distribution. 
\end{theorem}

We also recall the result of Hochman \cite[Proposition 1.9]{Hochmanpreprint} on the effect of pushing measures forward under diffeomorphisms on the generated distributions.

\begin{proposition}\label{prop-Hochman1.9}
  Let $\mu$ be a Radon measure on $\R^d$. Then for any diffeomorphism $f\colon\R^d\to\R^d$ we have
  \begin{equation*}
    \lim_{n\to\infty} d_{\rm LP} \biggl( \frac{1}{n}\int_0^n (Df|_x)_* \delta_{\mu_{x,t}}\dd t, \frac{1}{n}\int_0^n \delta_{(f_*\mu)_{f(x),t}}\dd t\biggr) = 0
  \end{equation*}
  for $\mu$-almost all $x$.
\end{proposition}

For a Radon measure $\mu$ the \emph{upper and lower pointwise dimensions} of $\mu$ at $x \in \R^d$ are
\begin{align*}
  \udimloc(\mu,x) &= \limsup_{r \to 0} \frac{\log \mu(B(x,r))}{\log r}, \\
  \ldimloc(\mu,x) &= \liminf_{r \to 0} \frac{\log \mu(B(x,r))}{\log r},
\end{align*}
respectively. For a given measure, the pointwise dimensions naturally introduce four different dimensions. For example, the \emph{lower Hausdorff dimension} of $\mu$ is $\ldimh\mu = \essinf_{x \sim \mu} \ldimloc(\mu,x)$. If there exists a constant $s$ such that $\udimloc(\mu,x) = \ldimloc(\mu,x) = s$ for $\mu$-almost all $x \in \R^d$, then we write $\dim\mu = s$ and say that $\mu$ is \emph{exact-dimensional}. For example, quasi-Bernoulli measures on self-conformal sets are exact-dimensional; see Feng and Hu \cite[Theorem 2.8]{FengHu2009}. Hochman \cite[Lemma 1.18]{Hochmanpreprint} showed that if $P$ is a fractal distribution, then $P$-almost every measure is exact-dimensional. This allows us to define the \emph{dimension} of a fractal distribution $P$ to be
\begin{equation*}
  \dim P = \int \dim\nu\dd P(\nu) = \int \frac{\log \nu(B(0,r))}{\log r} \dd P(\nu)
\end{equation*}
for all $0<r<1$. Although dimensions of measures are usually highly discontinuous, the function $P \mapsto \dim P$ defined on $\mathcal{FD}$ is continuous; see \cite[Lemma 3.20]{KaenmakiSahlstenShmerkin2015GMT}. Furthermore, Hochman \cite[Proposition 1.19]{Hochmanpreprint} (see also the remark after \cite[Theorem 3.21]{KaenmakiSahlstenShmerkin2015GMT}) showed that tangent distributions encode all information on dimensions:

\begin{proposition} \label{prop-Hochman1.19}
  Let $\mu$ be a Radon measure. Then
  \begin{align*}
    \udimloc(\mu,x) &= \sup\{\dim P : P \in \mathcal{FD} \text{ is a tangent distribution of $\mu$ at $x$}\}, \\ 
    \ldimloc(\mu,x) &= \inf\{\dim P : P \in \mathcal{FD} \text{ is a tangent distribution of $\mu$ at $x$}\},
  \end{align*}
  for $\mu$-almost all $x \in \R^d$. In particular, if $\mu$ is uniformly scaling and generates $P \in \mathcal{FD}$, then $\mu$ is exact-dimensional and $\dim\mu = \dim P$.
\end{proposition}

\section{Self-conformal measures are uniformly scaling} \label{section: proof-scenery-flow}

In this section, we prove Theorems \ref{thm-unif-scaling-1}--\ref{thm-unif-scaling-3}. Let $\mu \in \mathcal{P}(X)$ and $\mathcal{A}$ be a finite partition of $X$. The \emph{Shannon entropy} of $\mu$ with respect to $\mathcal{A}$ is
$$
  H(\mu,\mathcal{A}) = -\sum_{A\in\mathcal{A}}\mu(A)\log \mu(A).
$$
If, in addition, $\nu \in \mathcal{P}(X)$ is such that $\nu(A)=0$ implies $\mu(A)=0$ for all $A \in \mathcal{A}$, then the \emph{Kullback–Leibler divergence} between $\nu$ and $\mu$ with respect to $\mathcal{A}$ is
\begin{equation*}
  D_{\mathrm{KL}}(\mu \,\|\, \nu, \mathcal{A}) = -\sum_{A\in\mathcal{A}}\mu(A)\log \frac{\nu(A)}{\mu(A)}.
\end{equation*}
If the partition in use is clear from the context, then we omit it in notation. We recall the classical Gibbs' inequality.

\begin{lemma}[Gibbs' inequality] \label{lem:Gibbs}
  Let $\mathcal{I}$ be a finite set, and $(p_i)_{i\in\mathcal{I}}$ and $(q_i)_{i\in\mathcal{I}}$ probability vectors such that $q_i=0$ implies $p_i=0$. Then
  \begin{enumerate}
    \item\label{it:gibbs1} $D_{\mathrm{KL}}((p_i)_{i\in\mathcal{I}} \,\|\, (q_i)_{i\in\mathcal{I}}) \ge 0$,
    \item\label{it:gibbs2} $D_{\mathrm{KL}}((p_i)_{i\in\mathcal{I}} \,\|\, (q_i)_{i\in\mathcal{I}}) < \eps$ implies $\sum_{i\in\mathcal{I}}|p_i-q_i|<\sqrt{2\varepsilon\#\mathcal{I}}$ for all $\varepsilon>0$.
  \end{enumerate}
\end{lemma}

\begin{proof}
	It is well known that $-\log(1-x)\geq x+\frac{x^2}{2\max\{1,1-x\}}$ for all $x<1$. Hence,
	\begin{align*}
    D_{\mathrm{KL}}((p_i)_{i\in\mathcal{I}} \,\|\, (q_i)_{i\in\mathcal{I}}) &\geq \sum_{i\in\mathcal{I}}p_i\Biggl(\frac{p_i-q_i}{p_i}+\frac{\bigl(\frac{p_i-q_i}{p_i}\bigr)^2}{2\max\bigl\{1,1-\frac{p_i-q_i}{p_i}\bigr\}}\Biggr)\\
	  &= \sum_{i\in\mathcal{I}}\frac{(p_i-q_i)^2}{2\max\{p_i,q_i\}}\geq\frac{1}{2}\sum_{i\in\mathcal{I}}(p_i-q_i)^2,
	\end{align*}
	which implies the claims of the lemma.
\end{proof}

If $(p_i)_{i\in\mathcal{I}}$ and $(q_i)_{i\in\mathcal{I}}$ are probability vectors, then we write
\begin{equation*}
	\dist((p_i)_{i\in\mathcal{I}},(q_i)_{i\in\mathcal{I}}) = \sum_{i\in\mathcal{I}}|p_i-q_i|
\end{equation*}
and observe that
\begin{equation}\label{eq:forLP}
	\dist (( \mu(I))_{I\in \mathcal{D}_\ell}, ( \nu(I))_{I\in \mathcal{D}_\ell})\leq 2^{-\ell} \quad\Rightarrow\quad d_{\rm LP}(\mu,\nu)\leq 2^{-\ell}
\end{equation}
for all $\ell\in\N$. Let us next state two technical lemmas.

\begin{lemma}\label{lem:tech1}
  Let $\eta$ be a Radon measure with $\ldimh\eta>0$. Then for $\eta$-almost every $x$ and every $\eps>0$ there is $\rho>0$ such that
  \begin{equation} \label{eq:sup-of-difference}
	  \limsup_{N \to \infty} \frac{1}{N} \sum_{k=0}^{N-1} \sup_{\rho \le r \le 1} \eta_{x,k}(B(0,r+\rho) \setminus B(0,r-\rho)) < \eps.
  \end{equation}
\end{lemma}

\begin{proof}
Let us argue by contradiction: for some $\varepsilon>0$ and any $\rho>0$, we find a sequence $(N_k^\rho)_{k\in\N}$ along which the left-hand side of \eqref{eq:sup-of-difference} exists as a limit, is at least $\varepsilon$, and that the sequence $\frac{1}{N_k^\rho}\sum_{k=0}^{N_k^\rho-1}\delta_{\eta_{x,k}}$ converges to a tangent distribution $P^\rho$. Furthermore, by adapting \cite[proof of Proposition 5.5(3)]{Hochmanpreprint}, we see that $\int_0^1 S_tP^\rho\dd t$ is a fractal distribution and, by applying Markov's inequality, it gives mass at least $\varepsilon^2$ to measures which further give mass at least $\varepsilon^2$ to the $\rho$-neigbourhood of the boundary of a closed ball centered at the origin. Let $P$ be a weak$^*$ accumulation point of $P^\rho$ as $\rho\to 0$ and notice that, by \cite[Theorem 1.1]{KaenmakiSahlstenShmerkin2015}, $P$ is a fractal distribution. Moreover, 
\begin{align*}
	\dim P &= \iint_0^1 \frac{\log S_t\zeta(B(0,2^{-1}))}{\log 2^{-1}}\dd t \dd P(\zeta) \\
	&= \lim_{\rho\to 0} \iint_0^1 \frac{\log S_t \zeta(B(0,2^{-1}))}{\log 2^{-1}}\dd t \dd P^\rho(\zeta) \ge \ldimh\eta > 0
\end{align*}
since the inner Lebesgue integral is a continuous function of $\zeta$. Recalling that $P$ is supported on measures which give positive mass to a boundary of a ball, we see that this is impossible by \cite[Theorem 3.22]{KaenmakiSahlstenShmerkin2015GMT}.
\end{proof}

If $\mu$ is a Radon measure on $\R^d$ and $A \subseteq \R^d$ is a Borel set with $0<\mu(A)<\infty$, then we write $\mu_A = \mu(A)^{-1}\mu|_A$.

\begin{lemma}\label{lem:tech2}
For any $r,\tau>0$ the following holds for all small enough $\eps' \ge \rho > 0$: If $\eta$ and $\lambda$ are Borel probability measures on $[-1,1]^d$ and $B=B(x,r)$ is a closed ball such that
\begin{align*}
	d_{\mathrm{LP}}(\eta,\lambda) &< \rho, \\ 
	\lambda(B(x,r+\rho)) &\le \lambda(B(x,r-\rho)) + \eps', \\ 
	\min\{\eta(B),\lambda(B)\} &\ge \tau,
\end{align*}
then
\begin{equation*} 
	d_{\mathrm{LP}}(\eta_B,\lambda_B) < O(\eps'/\tau).
\end{equation*}
\end{lemma}

\begin{proof}
For any measurable set $A$, the assumptions on $\eta$ and $\lambda$ imply
\begin{align*}
	\eta_B(A) &= \eta(B)^{-1}\eta(B \cap A) \\ 
	&\le (\lambda(B(x,r-\rho))-\rho)^{-1} (\lambda(B^\rho \cap A^\rho)+\rho) \\ 
	&\le (\lambda(B)-2\eps')^{-1} \lambda(B \cap A^\rho) + 2\eps'(\lambda(B)-2\eps')^{-1} \\ 
	&\le \frac{1}{1-2\eps'/\tau} \lambda(B)^{-1} \lambda(B \cap A^\rho) + \frac{3\eps'}{\tau} \\ 
	&\le \lambda_B(A^{O(\eps'/\tau)}) + O(\eps'/\tau)
\end{align*}
and, similarly,
\begin{equation*}
	\lambda_B(A) \le \eta_B(A^{O(\eps'/\tau)}) + O(\eps'/\tau)
\end{equation*}
which verifies the claim. 
\end{proof}

The main ingredient in the proofs of Theorems \ref{thm-unif-scaling-1}--\ref{thm-unif-scaling-3} is the following proposition. We remark that the proposition does not assume any condition on the smallness of entropy or dimension of $\nu$. Write
\begin{equation*}
  \nu\cdot \mu = \int h_*\mu \dd\nu(h)
\end{equation*}
for all Borel probability measures $\mu$ on $[-1,1]^d$ and $\nu$ supported on $\mu$-measurable $\R^d \to \R^d$ functions.

\begin{proposition}\label{prop-main-1}
  Let $\mu$ and $\nu$ be Borel probability measures such that $\mu$ is supported on $[-1,1]^d$ and $\nu$ supported on $\mu$-measurable $\R^d \to \R^d$ functions. For any $\varepsilon>0$ there exist $N\in\N$ and $\delta>0$ such that the following holds for all $n\geq N$: If
  \begin{enumerate}
    \item\label{it:mainprop1} $h_*\mu (\{x : \frac{1}{n}\sum_{k=1}^n H((h_*\mu)^{\mathcal{D}_k(x)},\mathcal{D}_1)\ge \alpha\}) \ge 1-\delta$ for $\nu$-almost all $h$,
    \item\label{it:mainprop2} $H(\nu\cdot\mu,\mathcal{D}_n) \le n(\alpha+\delta)$,
    \item\label{it:mainprop3} $\ldimh\nu\cdot\mu > 0$,
  \end{enumerate}
  for some $\alpha\ge 0$, then 
  \begin{equation}\label{eq:prop-main-1:2}
    \iint \frac{1}{n}\int_1^n d_{\rm LP}((\nu\cdot\mu)_{y,t}, (h_*\mu)_{y,t})\dd t\dd (h_*\mu)(y)\dd\nu(h) < \varepsilon.
  \end{equation}
\end{proposition}

\begin{proof}
By elementary properties of entropy, we can write
\begin{equation}\label{eq:proof-Prop-main: 1}
H(\eta,\mathcal{D}_n)=-\sum_{k=0}^{n-1}\sum_{I_k\in \mathcal{D}_k}\sum_{\atop{I_{k+1}\in \mathcal{D}_{k+1}}{I_{k+1}\subset I_k}}\eta(I_{k+1})\log \frac{\eta(I_{k+1})}{\eta(I_k)}
\end{equation}
for all $\eta\in \mathcal{P}(\R^d)$. Applying \eqref{eq:proof-Prop-main: 1} to $\nu\cdot\mu$, we get 
\begin{align*}
H(\nu\cdot\mu,\mathcal{D}_n) &= -\sum_{k=0}^{n-1}\sum_{I_k\in \mathcal{D}_k}\sum_{\atop{I_{k+1}\in \mathcal{D}_{k+1}}{I_{k+1}\subset I_k}}\int \biggl(h_*\mu(I_{k+1})\log \frac{\nu\cdot\mu(I_{k+1})}{\nu\cdot\mu(I_k)}\biggr)\dd\nu(h) \\
&= -\sum_{k=0}^{n-1}\sum_{I_k\in \mathcal{D}_k}\int \biggl(h_*\mu(I_{k})\sum_{\atop{I_{k+1}\in \mathcal{D}_{k+1}}{I_{k+1}\subset I_k}}\frac{h_*\mu(I_{k+1})}{h_*\mu(I_{k})}\log \frac{\nu\cdot\mu(I_{k+1})}{\nu\cdot\mu(I_k)} \biggr)\dd\nu(h)\\
&= -\sum_{k=0}^{n-1} \iint \biggl(\sum_{I\in \mathcal{D}_1} (h_*\mu)^{\mathcal{D}_k(x)}(I) \log (\nu\cdot\mu)^{\mathcal{D}_k(x)}(I)\biggr)\dd(h_*\mu)(x)\dd\nu(h)\\
&= - \iint \sum_{k=0}^{n-1} \biggl(\sum_{I\in \mathcal{D}_1} (h_*\mu)^{\mathcal{D}_k(x)}(I) \log (\nu\cdot\mu)^{\mathcal{D}_k(x)}(I)\biggr)\dd(h_*\mu)(x)\dd\nu(h).
\end{align*}
For simplicity, write
\begin{align*}
  p^h_{k,\ell}(x) &= ( (h_*\mu)^{\mathcal{D}_k(x)}(I))_{I\in \mathcal{D}_\ell}, \\
  q_{k,\ell}(x) &= ( (\nu\cdot\mu)^{\mathcal{D}_k(x)}(I))_{I\in \mathcal{D}_\ell},
\end{align*}
in which case
\begin{align*}
  H(p_{k,1}^h(x)) &= -\sum_{I\in \mathcal{D}_1} (h_*\mu)^{\mathcal{D}_k(x)}(I) \log (h_*\mu)^{\mathcal{D}_k(x)}(I), \\ 
  D_{\mathrm{KL}}(p^h_{k,1}(x) \,\|\, q_{k,1}(x)) &= -\sum_{I\in \mathcal{D}_1} (h_*\mu)^{\mathcal{D}_k(x)}(I) \log \frac{(\nu\cdot\mu)^{\mathcal{D}_k(x)}(I)}{(h_*\mu)^{\mathcal{D}_k(x)}(I)}.
\end{align*}
This allows us to rewrite
\begin{equation} \label{eq:proof-Prop-main: 2}
\begin{split}
  H(\nu\cdot\mu,\mathcal{D}_n) &= \iint \sum_{k=0}^{n-1} H(p_{k,1}^h(x)) \dd(h_*\mu)(x)\dd\nu(h) \\ 
  &\qquad\qquad+\iint \sum_{k=0}^{n-1} D_{\mathrm{KL}}(p^h_{k,1}(x) \,\|\, q_{k,1}(x)) \dd(h_*\mu)(x)\dd\nu(h).
\end{split}
\end{equation}
Recall that, by Gibbs' inequality Lemma~\ref{lem:Gibbs}\eqref{it:gibbs1}, we have
$$
  D_{\mathrm{KL}}(p^h_{k,1}(x) \,\|\, q_{k,1}(x)) \ge 0.
$$
Since, by the assumption \eqref{it:mainprop1},
\begin{equation} \label{eq:nalpha-odelta}
  \int \sum_{k=0}^{n-1} H(p_{k,1}^h(x)) \dd(h_*\mu)(x) \ge n(\alpha-o_{\delta}(1)),
\end{equation}
we get, by combining \eqref{eq:proof-Prop-main: 2} with the assumption \eqref{it:mainprop2},
\begin{equation} \label{eq:KL-small-estimate}
\begin{split}
  \iint \sum_{k=0}^{n-1} D_{\mathrm{KL}}&(p^h_{k,1}(x) \,\|\, q_{k,1}(x))\dd(h_*\mu)(x)\dd\nu(h) \\
  &= H(\nu\cdot\mu,\mathcal{D}_n) - \iint \sum_{k=0}^{n-1} H(p_{k,1}^h(x)) \dd(h_*\mu)(x)\dd\nu(h) \leq o_\delta(1)n.  
\end{split}
\end{equation}
To simplify notation, write
\begin{equation*}
  \mathbb{E}(\,\cdot\,) = \iint \frac{1}{n}\sum_{k=0}^{n-1} \,\cdot\, \dd(h_*\mu)(x)\dd\nu(h).
\end{equation*}
Combining \eqref{eq:KL-small-estimate} and Gibbs' inequality Lemma~\ref{lem:Gibbs}\eqref{it:gibbs2} with Markov's inequality, we get
\begin{equation}\label{eq:forfinal}
\begin{split}
  \mathbb{E}(\mathrm{dist}&(p^h_{k,1}(x),q_{k,1}(x))) \\ &= \mathbb{E}(\mathrm{dist}(p^h_{k,1}(x),q_{k,1}(x))\mathds{1}\{D_{\mathrm{KL}}(p^h_{k,1}(x) \,\|\, q_{k,1}(x))>\sqrt{o_\delta(1)}\}) \\
  &\qquad\qquad+\mathbb{E}(\mathrm{dist}(p^h_{k,1}(x),q_{k,1}(x))\mathds{1}\{D_{\mathrm{KL}}(p^h_{k,1}(x) \,\|\, q_{k,1}(x))\leq\sqrt{o_\delta(1)}\})\\
  &\le 2\mathbb{E}(\mathds{1}\{D_{\mathrm{KL}}(p^h_{k,1}(x) \,\|\, q_{k,1}(x))>\sqrt{o_\delta(1)}\})+2^{d/2}\sqrt{2o_\delta(1)}\\
  &\le \frac{2\mathbb{E}(D_{\mathrm{KL}}(p^h_{k,1}(x) \,\|\, q_{k,1}(x)))}{\sqrt{o_\delta(1)}}+2^{d/2}\sqrt{2o_\delta(1)}\\
  &\le (2+2^{d/2}\sqrt{2})\sqrt{o_\delta(1)},
\end{split}
\end{equation}
where $\mathds{1}A$ is the indicator function of a set $A \subseteq \R^d$. Write $o'_\delta(1) = (2+2^{d/2}\sqrt{2})\sqrt{o_\delta(1)}$ and notice that $o'_\delta(1)\to 0$ as $\delta\to 0$. Simple calculations show that
\begin{align*}
  \iint \dist&(p^h_{k,\ell}(x),q_{k,\ell}(x)) \dd(h_*\mu)\dd\nu(h) \\ 
  &\le \sum_{m=0}^{\ell-1} \iint \dist(p^h_{k+m,1}(x),q_{k+m,1}(x)) \dd(h_*\mu)\dd\nu(h)
\end{align*}
for all $\ell \in \N$. Hence, for every $n \ge \ell$ large enough, \eqref{eq:forfinal} implies
\begin{align*}
  \sum_{k=0}^{n-1} \iint \dist&(p^h_{k,\ell}(x),q_{k,\ell}(x)) \dd(h_*\mu)\dd\nu(h) \\ 
  &\le \ell \sum_{k=0}^{n+\ell-1} \iint \dist(p^h_{k,1}(x),q_{k,1}(x)) \dd(h_*\mu)\dd\nu(h) \le 2o'_\delta(1)\ell n.
\end{align*}
Choose $\ell \in \N$ such that $2^{-(\ell+1)}<\sqrt{2o_\delta'(1)}\leq 2^{-\ell}$ and combine the above inequality with \eqref{eq:forLP} and Markov's inequality to obtain
\begin{equation*}
\begin{split}
  \mathbb{E}(d_{\mathrm{LP}}&((h_*\mu)^{\mathcal{D}_k(y)},(\nu\cdot\mu)^{\mathcal{D}_k(y)}))\\
  &=\mathbb{E}(d_{\mathrm{LP}}((h_*\mu)^{\mathcal{D}_k(y)},(\nu\cdot\mu)^{\mathcal{D}_k(y)})\mathds{1}\{\dist(p^h_{k,\ell}(x),q_{k,\ell}(x))>2^{-\ell}\})\\
  &\qquad\qquad+\mathbb{E}(d_{\mathrm{LP}}((h_*\mu)^{\mathcal{D}_k(y)},(\nu\cdot\mu)^{\mathcal{D}_k(y)})\mathds{1}\{\dist(p^h_{k,\ell}(x),q_{k,\ell}(x))\leq2^{-\ell}\})\\
  &\leq 2^{\ell}\mathbb{E}(\dist(p^h_{k,\ell}(x),q_{k,\ell}(x)))+2^{-\ell}\leq 2^{\ell}\ell2o_\delta'(1)+2^{-\ell}\\
  &\leq -\sqrt{2o_{\delta}'(1)}\log_2\sqrt{2o_{\delta}'(1)}+\sqrt{2o_{\delta}'(1)}.
\end{split}
\end{equation*}
In other words, denoting the above upper bound by $o_\delta''(1)$, we have obtained
\begin{equation}\label{eq:prop-main-1:1}
   \iint \frac{1}{n}\sum_{k=1}^n d_{\rm LP}((\nu\cdot\mu)^{\mathcal{D}_k(y)}, (h_*\mu)^{\mathcal{D}_k(y)}) \dd (h_*\mu)(y)\dd\nu(h) < o_\delta''(1),
\end{equation}
where $o_\delta''(1)\to 0$ as $\delta\to 0$. To conclude \eqref{eq:prop-main-1:2}, it therefore remains to transition from dyadic magnifications to magnifications along balls. 

To finish the proof, let us show how \eqref{eq:prop-main-1:1} implies \eqref{eq:prop-main-1:2}. Fix $\delta>0$ and $m\in\N$. By randomly translating $\nu\cdot \mu$ with respect to the Lebesgue measure, we may assume that $\nu\cdot \mu$-almost every $y$ equidistributes for the Lebesgue measure under the map 
$$
  F \colon \R^d \to \R^d, \quad  F(x_1,\ldots, x_d) = (2x_1 \mod 1, \ldots, 2x_d\mod 1).
$$ 
In particular, for $\nu\cdot\mu$-almost every $y$, there exists a set $\mathcal{N}_\delta\subseteq \N$ such that
\begin{equation*}
  \liminf_{n\to\infty}\frac{1}{n}\#(\mathcal{N}_\delta\cap [0,n]) \geq 1-\delta
\end{equation*}
and $B(F^{k} y, 2^{-m+1}) \subseteq [0,1]^d$ or, equivalently,
\begin{equation}\label{cor-1-1-eq2}
    B(y,2^{-k-m+1})\subseteq \mathcal{D}_k(y)
\end{equation}
for all $k\in \mathcal{N}_\delta$. For any Borel measure $\eta$ and $\eta$-almost every $y$, it follows from Proposition 
\ref{prop-Hochman1.19} that
\begin{equation*}
    \limsup_{n\to\infty} \frac{1}{n}\sum_{k=1}^n \frac{\log \eta_{y,k}(B(0,2^{-m}))}{m} \leq \overline{\dim}_{\rm loc}(\eta, y).
\end{equation*}
Applying this inequality, we see that for $\nu$-almost every $h$ and for $h_*\mu$-almost every $y$ there exists a number $c>0$ and a set $\mathcal{N}_\delta'\subseteq \N$ such that $\liminf_{n\to\infty}\frac{1}{n}\#(\mathcal{N}_\delta'\cap [0,n]) \geq 1-\delta$ and
\begin{equation*}
    c\leq \frac{\eta(B(y,2^{-k-m}))}{\eta(B(y,2^{-k}))}\leq 1
\end{equation*}
for all $k\in \mathcal{N}_\delta'$ and for both $\eta \in \lbrace \nu\cdot \mu, h_*\mu\rbrace$. In particular, for each $k\in \mathcal{N}_\delta \cap \mathcal{N}_\delta'$, we have
\begin{equation}\label{cor-1-1-eq1}
  c \leq \frac{\eta(B(y,2^{-k-m}))}{\eta(\mathcal{D}_k(y))}\leq 1
\end{equation}
for both $\eta \in \lbrace \nu\cdot \mu, h_*\mu\rbrace$.

Let $\varepsilon'>0$. By the assumption \eqref{it:mainprop3} and Lemma~\ref{lem:tech1}, there exists $\rho>0$ such that for $\nu\cdot\mu$-almost every $y$, there exists a set $\mathcal{N}_\delta''\subseteq\N$ such that $\liminf_{n\to\infty} \frac{1}{n} \#(\mathcal{N}_\delta'' \cap [0,n]) \geq 1-\delta$ and 
\begin{equation}\label{cor-1-1-eq3}
  \sup_{\rho \le r \le 1} (\nu\cdot\mu)_{y,k}(B(0, r+\rho) \setminus B(0,r-\rho)) < \varepsilon'
\end{equation}
for all $k\in \mathcal{N}_\delta''$. Replacing $\rho$ by a smaller number, if necessary, and applying Lemma~\ref{lem:tech2}, we conclude that for every $y$ and $k$ satisfying \eqref{cor-1-1-eq2}, \eqref{cor-1-1-eq1}, \eqref{cor-1-1-eq3}, and
\begin{equation*}
  d_{\rm LP}((\nu\cdot\mu)^{\mathcal{D}_k(y)},  (h_*\mu)^{\mathcal{D}_k(y)}) < \rho,
\end{equation*}
we have
\begin{equation*}
  \int_0^1 d_{\rm LP} ((\nu\cdot\mu)_{y,k+m-1+t},  (h_*\mu)_{y,k+m-1+t})\dd t < O(\varepsilon'/c)
\end{equation*}
for every $0\leq t \leq 1$. Choosing $\varepsilon'$ small enough, the quantity $O(\varepsilon'/c)$ can be taken arbitrarily small independently of $\varepsilon$. In particular, if we choose $\delta$ small enough, then we have the implication 
\begin{align*}
    \frac{1}{n}\sum_{k=1}^n d_{\rm LP}((\nu\cdot\mu)^{\mathcal{D}_k(y)},  (h_*\mu)^{\mathcal{D}_k(y)}) < \rho^2 
    \quad\Rightarrow\quad \frac{1}{n}\int_m^{n+m}d_{\rm LP} ((\nu\cdot\mu)_{y,t},  (h_*\mu)_{y,t})\dd t < \varepsilon
\end{align*}
for all large enough $n$. Since this implication holds for $\nu$-almost every $h$ and $h_*\mu$-almost every $y$, the claim \eqref{eq:prop-main-1:2} follows from the already established \eqref{eq:prop-main-1:1}.
\end{proof}

We will use Proposition \ref{prop-main-1} to show that a quasi-Bernoulli measure $\mu$ shares the same tangential behaviour with its tangent measures, that is, its weak$^*$-accumulation points of the scenery $(\mu_{x,t})_{t\geq0}$. To that end, we will show in the following lemma that the tangents of a quasi-Bernoulli measure $\mu$ are equivalent with convex combinations of conformal images of $\mu$ itself.

\begin{lemma} \label{lemma-structureoftangents}
  Let $\Phi$ be a conformal iterated function system on $\R^d$ with attractor $K$ and $\mu$ be a quasi-Bernoulli measure on $K$. For $\mu$-almost every $x$ and every tangent distribution $P$ at $x$, for $P$-almost every $\eta$ there exists a measure $\nu \in \mathcal{P}(\mathcal{C}^{1+\alpha}(\R^d))$ such that $\eta \sim_C (\nu \cdot \mu)_{B(0,1)} $, where $C>0$ is as in \eqref{eq:quasi-bernoulli-def}.
\end{lemma}

\begin{proof}
Write $\Phi = (f_i)_{i \in \Lambda}$ and let $\bar{\mu}$ be a quasi-Bernoulli measure on $\Lambda^\N$ such that $\Pi_*\bar{\mu} = \mu$. Define
\begin{align*}
  \mathcal{C}^{1+\alpha}_{\rho, M}(B(0,1)) = \{ f\in \mathcal{C}^{1+\alpha}(\R^d) : \;&M^{-1}\Vert Df|_y\Vert \leq \Vert D f|_x\Vert \leq M\Vert Df|_y\Vert \\ &\text{and } \Vert Df|_x\Vert \geq \rho \text{ for all } x,y\in {\rm Dom}(f), \\
  &\text{and } f(B(0,1))\cap B(0,1)\neq\emptyset \}
\end{align*}
for all $\rho>0$ and $M \ge 1$. For the rest of the proof, fix $\rho>0$ and $M \ge 1$ such that $\mathcal{C}^{1+\alpha}_{\rho,M}(B(0,1))$ contains the functions $f_i$ of the iterated function system $\Phi$ and furthermore, relying on the bounded distortion property, $\mathcal{C}^{1+\alpha}_{0,M}(B(0,1))$ contains all the compositions of $f_i$. We remark that, by the Arzel{\'a}-Ascoli theorem, the set $\mathcal{C}^{1+\alpha}_{\rho,M}(B(0,1))$ is compact in the topology induced by the norm $\Vert f \Vert_{\mathcal{C}^{1}} = \Vert f\Vert_\infty + \Vert Df \Vert_{\infty}$.

Define
\begin{align*}
\Lambda(x,t) = \lbrace \mathtt{i}\in\Lambda^* : \;&\diam(f_\mathtt{i}(B(0,1)))\leq 2^{-t} < \diam(f_{\mathtt{i^-}}(B(0,1))) \\ 
&\,\text{and } f_\mathtt{i}(B(0,1))\cap B(x,2^{-t})\neq\emptyset \rbrace
\end{align*}
for all $x \in \spt\mu$ and $t>0$. It follows from the quasi-Bernoulli property \eqref{eq:quasi-bernoulli-def} that 
\begin{equation}\label{eq:lemma3.3-quasibernoulli}
    \mu|_{B(x,2^{-t})} \sim_C \sum_{{\mathtt{i}} \in \Lambda(x,t)} \bar{\mu}([\mathtt{i}]) (f_{\mathtt{i}})_* \mu|_{B(x,2^{-t})},
\end{equation}
where $C>0$ is as in \eqref{eq:quasi-bernoulli-def}.
Let $g_{x,t}$ be the affine map $y\mapsto 2^t (y-x)$ and write 
\begin{equation*}
\nu(x,t) = C(x,t)^{-1} \sum_{{\mathtt{i}} \in \Lambda(x,t)} \bar{\mu}([\mathtt{i}]) \delta_{g_{x,t}\circ f_{\mathtt{i}}} \in \mathcal{P}(\mathcal{C}^{1+\alpha}_{\rho,M}(\R^d)),
\end{equation*}
where $C(x,t) = \sum_{\mathtt{i}\in\Lambda(x,t)} \bar{\mu}([\mathtt{i}])$. Notice that, by \eqref{eq:lemma3.3-quasibernoulli}, we have $\mu_{x,t}\sim_C (\nu(x,t)\cdot\mu)_{B(0,1)}$. Define
\begin{align*}
    E &= \lbrace \eta\in\mathcal{P}(B(0,1)) : \eta\sim_C (\nu\cdot\mu)_{B(0,1)} \text{ for some } \nu\in \mathcal{P}(\mathcal{C}^{1+\alpha}_{\rho,M}(B(0,1)))\rbrace, \\
    E_r &= \lbrace \eta\in\mathcal{P}(B(0,1)) : \nu\cdot\mu(B(0,1))\geq r \text{ and} \\ &\qquad\qquad\qquad\qquad\quad \eta\sim_C (\nu\cdot\mu)_{B(0,1)} \text{ for some } \nu\in \mathcal{P}(\mathcal{C}^{1+\alpha}_{\rho,M}(B(0,1)))\rbrace
\end{align*}
for all $r>0$ and note that $\mu_{x,t}\in E$ for all $x \in \spt\mu$ and $t\geq 0$. Recall that our aim is to show that for $\mu$-almost every $x$ and any tangent distribution $P$ at $x$, we have $P(E)=1$. Notice that $P(E) = P(\bigcup_{r>0} E_r)$ since $P$-almost every measure is supported on $B(0,1)$.

We claim that the set $E_r$ is closed for all $r>0$. Let $(\eta_n)_n$ be a converging sequence in $E_r$, with limit $\eta$. Using the compactness of $\mathcal{P}(\mathcal{C}^{1+\alpha}_{\rho, M}(B(0,1)))$ and by possibly passing to a subsequence, we find a converging sequence $(\nu_n)_n$ in $\mathcal{P}(\mathcal{C}^{1+\alpha}_{\rho,M}(B(0,1)))$ such that $\eta_n \sim_C (\nu_n\cdot\mu)_{B(0,1)}$ and $\nu_n\cdot\mu(B(0,1))\geq r$; let $\nu$ be the limit of this sequence. It follows from Lemma \ref{lem:tech2} and the continuity of the map $\nu\mapsto\nu\cdot\mu$ that $(\nu_n\cdot\mu)_{B(0,1)} \to (\nu\cdot\mu)_{B(0,1)}$. It is readily verified that $\nu\cdot\mu(B(0,1))\geq r$ and that $\eta \sim_C (\nu\cdot\mu)_{B(0,1)}$, whence $\eta\in E_r$. Thus $E_r$ is closed.

By \eqref{eq:lemma3.3-quasibernoulli}, we have 
\begin{equation}\label{eq:lemma3.3-zetaB}
    \nu(x,t)\cdot\mu(B(0,1)) = \frac{\mu(B(x,2^{-t}))}{C(x,t)}. 
\end{equation}
On the other hand, by the definition of $\Lambda(x,t)$ and the choice of $\rho$, we have
\begin{equation*}
    C(x,t)\leq \mu(B(x,\rho^{-1} 2^{-t}))
\end{equation*}
and, by Proposition \ref{prop-Hochman1.19},
\begin{equation*}
    \limsup_{n\to\infty} \frac{1}{n}\int_0^n -\log \frac{\mu(B(x,2^{-t}))}{\mu(B(x,\rho^{-1} 2^{-t}))}\dd t \leq \dim \mu
\end{equation*}
for $\mu$-almost every $x$. In particular, for $\mu$-almost every $x$ and any $\delta>0$, there exists $r>0$ and an open set $I_\delta\subseteq \R$ such that $\liminf_{n\to\infty} \frac{1}{n}\LL^1(I_\delta \cap [0,n]) \geq 1-\delta$, where $\LL^1$ is the Lebesgue measure on $\R$, and
\begin{equation*}
  \frac{\mu(B(x,2^{-t}))}{C(x,t)}\geq \frac{\mu(B(x,2^{-t}))}{\mu(B(x,\rho^{-1} 2^{-t}))} \geq r
\end{equation*}
for all $t\in I_\delta$. Now, for $\mu$-almost every $x$ and any tangent distribution $P$ at $x$, we have
\begin{align*}
    P(E) &\geq P(E_r) \geq \limsup_{n\to\infty}\frac{1}{n}\int_0^n \delta_{\mu_{x,t}} (E_r)\dd t \\
    &\geq \liminf_{n\to\infty}\frac{1}{n}\int_{I_\delta\cap[0,n]} \delta_{\mu_{x,t}} (E_r)\dd t 
    \geq 1-\delta
\end{align*}
as $E_r$ is closed and $\mu_{x,t}\in E_r$ for each $t\in I_\delta$. By letting $\delta \to 0$, this concludes the proof.
\end{proof}

If the conformal iterated function system satisfies the open set condition (or more generally, the weak separation condition), then the measure $\nu$ in Lemma \ref{lemma-structureoftangents} is a finite sum of Dirac masses. In general, this need not be the case and $\nu$ may even be non-atomic.

As indicated by Theorems \ref{thm-Hochman1.6} and \ref{thm-Hochman1.7}, tangent measures of $\mu$ possess much more regular local statistics than $\mu$ itself. The advantage of the representation given by Lemma \ref{lemma-structureoftangents} is that using Proposition \ref{prop-main-1}, we are able to transfer this regularity back to $\mu$. Combining Theorems \ref{thm-Hochman1.6} and \ref{thm-Hochman1.7}, Lemma \ref{lemma-structureoftangents}, and the Lebesgue-Besicovitch differentiation theorem, we deduce that for some $\eta \in \mathcal{P}(\mathcal{C}^{1+\alpha}(\R^d))$, the measure $\eta \cdot \mu$ is uniformly scaling and generates an ergodic fractal distribution. In order to apply Proposition \ref{prop-main-1} to deduce that also $\mu$ is uniformly scaling, we have to know that convolving $\mu$ with $\eta$ does not increase its entropy. However, this is ensured by Proposition \ref{prop-Hochman1.19} together with exact-dimensionality of $\mu$.

We are ready to prove Theorems \ref{thm-unif-scaling-1} and \ref{thm-unif-scaling-2}, beginning with the latter and afterwards restricting to the more special setting of the former. The proof of Theorem \ref{thm-unif-scaling-3} is contained in the proof of Theorem \ref{thm-unif-scaling-2}.

\begin{proof}[Proof of Theorems \ref{thm-unif-scaling-2} and \ref{thm-unif-scaling-3}]
We assume that $\dim \mu>0$; otherwise $\mu$ is uniformly scaling and generates the distribution $\delta_{\delta_0}$ by \cite[Corollary 6.6]{Hochmanpreprint}. Let $P_1$ and $P_2$ be ergodic components of tangent distributions of $\mu$, both of which satisfy the conclusion of Lemma \ref{lemma-structureoftangents}. Applying Theorem \ref{thm-Hochman1.6} and Lemma \ref{lemma-structureoftangents} to $P_1$ and $P_2$, we find measures $\nu_1, \nu_2\in\mathcal{P}(\mathcal{C}^{1+\alpha}(\R^d))$ such that $(\nu_i\cdot\mu)_{B(0,1)}$ is uniformly scaling and generates $P_i$ for both $i \in \{1,2\}$. In particular, by the dominated convergence theorem,
\begin{equation}\label{eq-P_i}
    \lim_{n\to\infty} \iint_{h^{-1}(B(0,1))} d_{\rm LP}\biggl(P_i, \frac{1}{n} \int_0^n \delta_{(\nu_i\cdot\mu)_{h(x), t}}\dd t\biggr) \dd\mu(x)\dd\nu_i(h) = 0.
\end{equation}
By Proposition \ref{prop-Hochman1.19}, the measures $\nu_1$ and $\nu_2$ can be chosen such that $(\nu_i \cdot \mu)_{B(0,1)}$ are exact-dimensional and $\lim_{n\to\infty} \frac{1}{n} H((\nu_i\cdot\mu)_{B(0,1)},\mathcal{D}_n) = \dim\mu > 0$ for both $i \in \{1,2\}$. Moreover, by e.g. \cite[Lemma 6.7]{Hochmanpreprint}, we have $\liminf_{n\to\infty}\frac{1}{n}\sum_{k=1}^n H((h_*\mu)^{\mathcal{D}_k(x)},\mathcal{D}_1) = \dim \mu$ for $\nu_i$-almost all $h$ and $h_*\mu$-almost all $x$, for both $i\in \lbrace 1,2\rbrace$. In particular, the measures $\nu_i$ and $\mu$ satisfy the conditions of Proposition \ref{prop-main-1}, and, by relying on that proposition, we see that for any $\varepsilon > 0$ and for all large enough $n$, we have 
\begin{equation}\label{eq-mu}
    \iint d_{\rm LP}\biggl( \frac{1}{n}\int_0^n \delta_{(\nu_i\cdot\mu)_{h(x),t}}\dd t, \frac{1}{n}\int_0^n \delta_{(h_*\mu)_{h(x),t}}\dd t\biggr)\dd\mu(x)\dd\nu_i(h) <\varepsilon.
\end{equation}
Combining \eqref{eq-P_i} and \eqref{eq-mu} and applying the triangle inequality, we obtain
\begin{equation}\label{eq-averageusc}
    \lim_{n\to\infty} \iint_{h^{-1}(B(0,1))} d_{\rm LP}\biggl(P_i, \frac{1}{n}\int_0^n \delta_{(h_*\mu)_{h(x),t}}\dd t\biggr) \dd\mu(x)\dd\nu_i(h) = 0
\end{equation}
for both $i \in \{1,2\}$. Choosing here a subsequence so that the integrand tends to $0$ pointwise proves Theorem \ref{thm-unif-scaling-3}.

Continuing with the proof of Theorem \ref{thm-unif-scaling-2}, we may combine \eqref{eq-averageusc} with Proposition \ref{prop-Hochman1.9} and the dominated convergence theorem to we obtain
\begin{equation}
    \lim_{n\to\infty} \iint_{h^{-1}(B(0,1))} d_{\rm LP}\biggl((D h|_x)^{-1}_*P_i, \frac{1}{n}\int_0^n \delta_{\mu_{x,t}}\dd t\biggr) \dd\mu(x)\dd\nu_i(h) = 0
\end{equation}
for both $i\in\lbrace 1,2\rbrace$. It can be seen from the proof of Lemma \ref{lemma-structureoftangents} that both $\nu_1$ and $\nu_2$ give positive measure to functions $h$ with $h(B(0,1))\subseteq B(0,1)$. In particular, for $\nu_i$-almost every $h$ for which $h(B(0,1))\subseteq B(0,1)$ there is a subsequence $(n_k)_{k\in\N}$ such that
\begin{equation*}
    \lim_{k\to\infty} d_{\rm LP}\biggl( (Dh|_x)^{-1}_*P_1, \frac{1}{{n_k}} \int_0^{n_k} \delta_{\mu_{x,t}}\dd t\biggr) = 0
\end{equation*}
for $\mu$-almost all $x$, for both $i\in\lbrace 1,2\rbrace$. Finally, an application of the triangle inequality yields that, for $\nu_1\times\nu_2$-almost every $(h,g)$ for which $h(B(0,1))\subseteq B(0,1)$ and $g(B(0,1))\subseteq B(0,1)$, we have
\begin{align*}
    &d_{\rm LP}((Dh|_x)^{-1}_*P_1, (Dg|_x)^{-1}_*P_2) \\
    &\qquad\qquad \leq\lim_{k\to\infty} d_{\rm LP}\biggl( (Dh|_x)^{-1}_*P_1, \frac{1}{{n_k}} \int_0^{n_k} \delta_{\mu_{x,t}}\dd t\biggr) \\
    &\qquad\qquad\qquad\qquad+ \lim_{k\to\infty}  d_{\rm LP}\biggl( (Dg|_x)^{-1}_*P_2, \frac{1}{{n_k}} \int_0^{n_k} \delta_{\mu_{x,t}}\dd t\biggr) 
    = 0
\end{align*}
or, equivalently, 
\begin{equation}\label{eq-P_1P_2}
    P_1 = (Dh|_x (Dg|_x)^{-1})_* P_2.
\end{equation}
Let now $Q$ be a tangent distribution of $\mu$ which is also a fractal distribution, and let $Q = \int P(\omega)\,dQ'(\omega)$ be its ergodic decomposition. By \eqref{eq-P_1P_2}, for $Q'$-almost every $\omega$ there exists an orthogonal matrix $O(\omega)$ such that $P(\omega) = O(\omega)_* P_2$. This induces a $Q'$-measurable function $\omega \mapsto O(\omega)$, and if we let $\zeta$ be the push-forward of $Q'$ under this map, we have obtained $Q = \int O_* P_2\dd \zeta(O)$ which is what we set out to prove.
\end{proof}

The proof of Theorem \ref{thm-unif-scaling-1} is built upon the above arguments.

\begin{proof}[Proof of Theorem \ref{thm-unif-scaling-1}]
Let us first assume that $d=1$. Since in this case $h$ and $g$ are (orientation-preserving) real functions, we have $Dh|_x = S_{-\log h'(x)}$ and $(Dg|_x)^{-1} = S_{\log g'(x)}$. Thus, by \eqref{eq-P_1P_2} and the $(S_t)_{t \ge 0}$-invariance of tangent distributions, $P_1 = P_2$ and $\mu$ is uniformly scaling. 

Suppose then that $d\geq 2$ and $\Phi$ consists of similarities. Then it is clear from the proof of Lemma \ref{lemma-structureoftangents} that the measures $\nu_1$ and $\nu_2$ in the proof of Theorem \ref{thm-unif-scaling-2} are supported on functions of the form $h(x) = r O_h(x) + a$ for some contraction $r>0$, a translation vector $a\in \R^d$, and an orthogonal matrix $O_h\in\mathcal{O}$, where $\mathcal{O}$ is the topological group generated by the orthogonal parts of the similarities in $\Phi$. In particular, if $P_1$ and $P_2$ are ergodic components of tangent distributions of $\mu$ outside a set of zero $\mu$-measure, then 
\begin{equation}\label{eq-P_1P_2selfsimilar}
    P_1 = (O_h O_g^{-1})_* P_2
\end{equation}
by \eqref{eq-P_1P_2}. 

To conclude that $\mu$ is uniformly scaling we reason as follows: Let $\varepsilon>0$ and let $\Lambda'\subseteq\Lambda^*$ be such that $\lbrace [\mathtt{i}]: \mathtt{i}\in\Lambda'\rbrace$ is a partition of $\Lambda^\N$, and the orthogonal parts of $f_\mathtt{i}$, $\mathtt{i}\in\Lambda'$, are $\varepsilon$-dense in $\mathcal{O}$. Using the quasi-Bernoulli property of $\mu$ to write $\nu_i\cdot \mu \sim_C \nu_i \cdot (\sum_{\mathtt{i}\in\Lambda'} \bar{\mu}([\mathtt{i}]) \delta_{f_\mathtt{i}}\cdot \mu)$ and ``absorbing'' the measure $\sum_{\mathtt{i}\in\Lambda'} \bar{\mu}([\mathtt{i}]) \delta_{f_\mathtt{i}}$ into $\nu_i$, we may suppose that $(\nu_i\cdot\mu)_{B(0,1)}$ still generates $P_i$ and that the set $\lbrace O_h: h\in\spt\nu_i, h(B(0,1))\subseteq B(0,1)\rbrace$ is $\varepsilon$-dense in $\mathcal{O}$. Thus, in \eqref{eq-P_1P_2selfsimilar}, we can choose $h$ and $g$ so that $\Vert O_h O_g^{-1} - {\rm Id}\Vert < \varepsilon$ and, consequently, $d_{\rm LP}(P_1, P_2) < 2\varepsilon$. Letting $\varepsilon\to 0$ completes the proof.
\end{proof}

\section{Normal numbers in self-conformal sets}\label{section-normalnumbers}

In this section, we prove Theorem \ref{pointwise_normality}. We begin by recalling some of the definitions and results of Hochman and Shmerkin \cite{HochmanShmerkin2015} which are used to study spectral properties of generated distributions. As in the proof of \cite[Theorem 1.4]{HochmanShmerkin2015}, Theorem \ref{pointwise_normality} is proven through an application of \cite[Theorem 1.1]{HochmanShmerkin2015} combined with an analysis of the pure-point spectrum of the generated distribution. Let $P$ be a distribution invariant under $(S_t)_{t \ge 0}$ and write $e(x) = e^{2 \pi i x} \in \C$ for all $x \in \R$. We call a number $\alpha \geq 0$ an \emph{eigenvalue} of $P$ if there exists a non-trivial measurable function $\varphi: \mathcal{P}(\R^d) \to \C$ such that 
$$
  \varphi(S_t\mu) = e(t\alpha)\varphi(\mu)
$$
for every $t \geq 0$ and for $P$-almost all $\mu$. Such a function $\varphi$ is called an \emph{eigenfunction} for the eigenvalue $\alpha$. The collection of all eigenvalues of $P$ is called its \emph{pure-point spectrum}. Given a Radon measure $\mu$ and $t_0 > 0$, we say that $\mu$ $t_0$-\emph{generates} a distribution $Q$ at $x$ if 
$$
  Q = \lim_{n \to \infty} \frac{1}{n} \sum_{k=1}^n \delta_{\mu_{x,kt_0}}.
$$
We recall the following property of fractal distributions from \cite[Lemma 4.9 and Proposition 4.15]{HochmanShmerkin2015}. 

\begin{lemma} \label{lemma-tangentt0}
    If $P$ is a fractal distribution, then for any $t_0>0$, $P$-almost every $\eta$ $t_0$-generates an $S_{t_0}$-ergodic distribution $P_x$ at $\eta$-almost every $x$. Furthermore, if $\varphi$ is an eigenfunction of $P$ for some eigenvalue $k/t_0$, then there exists $c \in \C$ such that $\varphi(\zeta) = c$ for $\eta$-almost every $x$ and $P_x$-almost all $\zeta$.
\end{lemma}

We are now ready to prove Theorem \ref{pointwise_normality}. We will use Proposition \ref{prop-main-1} to relate the sceneries of $\mu$ and its tangent measures, and otherwise proceed exactly as in \cite[proof of Theorem 1.4]{HochmanShmerkin2015}

\begin{proof}[Proof of Theorem \ref{pointwise_normality}]
Let $P$ be the tangent distribution generated by $\mu$ almost everywhere. Let $\beta > 1$ be a Pisot number such that $\Phi$ is arithmetically independent of $\beta$. We will show that $\frac{k}{\log \beta}$ does not belong to the pure-point spectrum of $P$ for any integer $k \neq 0$. By \cite[Proposition 4.1]{HochmanShmerkin2015}, this is equivalent to $P$ being ergodic under the map $S_{\log \beta}$. The proof is then concluded by \cite[Theorem 1.1 and Theorem~1.2]{HochmanShmerkin2015}.

Suppose for a contradiction that $\frac{k}{\log \beta}$ does belong to the pure-point spectrum of $P$ for some $k \neq 0$. Write $t_0$ for the eigenvalue $\frac{k}{\log \beta}$ and let $\varphi \colon \mathcal{P}(\R^d) \to \C$ be an associated eigenfunction. By Lemma \ref{lemma-tangentt0}, there exists a constant $c$ such that, for $P$-almost every $\eta$ and for $\eta$-almost every $x$, $\eta$ $t_0$-generates a distribution $P_x$ such that $\varphi(\zeta) = c$ for $P_x$-almost all $\zeta$. Since, by Lemma \ref{lemma-structureoftangents}, we can choose $\eta$ so that $\eta\sim_C (\nu\cdot\mu)_{B(0,1)}$ for some $\nu\in \mathcal{P}(\mathcal{C}^{1+\alpha}(\R^d))$, we see, arguing similarly as in the proof of Theorem \ref{thm-unif-scaling-3}, that for some $h\in\mathcal{C}^{1+\alpha}(\R^d)$, there exists a sequence $(n_k)_{k\in\N}$ such that
\begin{equation*}
    \lim_{k\to\infty} \frac{1}{n_k} \sum_{k=1}^{n_k} \delta_{(h_*\mu)_{h(x),kt_0}} = P_{h(x)}
\end{equation*}
for $\mu$-almost all $x$. In particular, along the subsequence $(n_k)_{k\in\N}$, $h_*\mu$ $t_0$-generates almost everywhere a distribution for which $\varphi \equiv c$ almost surely. Since $h_* \mu$ is a quasi-Bernoulli measure for the conjugated conformal iterated function system $h\Phi = \lbrace h \circ f_i \circ h^{-1}:\ i \in \Lambda \rbrace$ for which $\lambda (h\circ f_i\circ h^{-1}) = \lambda (f_i)$, without loss of generality, by switching the iterated function system in the beginning, we may assume that this is the case for $\mu$, that is,
\begin{equation*}
    \lim_{k\to\infty} \frac{1}{n_k} \sum_{k=1}^{n_k} \delta_{\mu_{x,kt_0}} = P_{x}
\end{equation*}
for $\mu$-almost every $x$. 

We now proceed to conclude the proof by relying on the assumption that $\Phi$ is arithmetically independent of $\beta$. Recall that $t_0 = \frac{k}{\log\beta}$. Using the independence assumption, we find $n\in\N$ and $\iii \in\Lambda^*$ such that $0<|\log \lambda(f_\iii) + n\log\beta| < 1/t_0$. Equivalently, $0<|t_0 \log \lambda(f_\iii) + kn| < 1$. In particular, $t_0\log\lambda(f_\iii)$ is not an integer, whence
\begin{equation}\label{not1}
  e(-t_0\log \lambda(f_\iii)) \neq 1.
\end{equation}
Let $x_0$ be the fixed point of $f_1$ and let $U$ be a small interval centered at $x_0$. Since $\mu_U \ll \mu$ and $(f_\iii)_*\mu_U \ll \mu$, it follows from the Lebesgue-Besicovitch differentiation theorem that also $\mu_U$ and $(f_\iii)_*\mu_U$ $t_0$-generate the distribution $P_x$ along the sequence $(n_k)_{k\in\N}$, for $\mu$-almost all $x$. On the other hand, by Proposition \ref{prop-Hochman1.9} (more precisely, see \cite[Lemma 4.16]{HochmanShmerkin2015} for the version for $t_0$-generated distributions), the distribution $t_0$-generated by $(f_\iii)_*\mu_U$ along the subsequence $(n_k)_{k\in\N}$ at $f_\iii(y)$ is $S_{-\log f'_\iii(y)} P_y$. In particular,
\begin{align*}
    c &= \frac{1}{\mu(U)}\int_{f_\iii(U)} \int \varphi(\zeta)\dd P_y(\zeta)\dd(f_\iii)_*\mu(y)\\
    &= \frac{1}{\mu(U)}\int_U \int \varphi(S_{-\log f'_\iii(y)} \zeta)\dd P_y(\zeta) \dd\mu(y)\\
    &= \frac{1}{\mu(U)}\int_U \int e(-t_0 \log f'_\iii(y)) \phi(\zeta)\dd P_y(\zeta)\dd \mu(y)\\
    &\to e(-t_0\log \lambda(f_\iii))c
\end{align*}
as $\diam(U) \to 0$. In light of \eqref{not1}, this is a contradiction. With the same modifications, the proof in the case that $\Phi$ is totally non-linear goes through exactly as in \cite[proof of Theorem 1.5]{HochmanShmerkin2015}.
\end{proof}

\section{Dissonance of quasi-Bernoulli measures}\label{section-dissonance}

In this section, we prove Theorem \ref{prop-dissonance}. Write $\Pi_{d,k}$ for the set of linear maps from $\R^d$ to $\R^k$. If $P$ is a fractal distribution, then, by \cite[Theorem 1.22]{Hochmanpreprint}, for every $\pi \in \Pi_{d,k}$ and for $P$-almost every $\mu$, the push-forward measure $\pi_*\mu$ is exact-dimensional. This allows us to define
\begin{equation*}
  E_P(\pi) = \int \dim \pi_*\mu \dd P(\mu)
\end{equation*}
for all fractal distributions $P$ and linear maps $\pi \in \Pi_{d,k}$. We denote the general linear group of degree $k$ by $GL_k(\R)$. Notice that there is an action of $GL_d(\R)$ on $\Pi_{d,k}$ given by $U$ such that $\pi \mapsto \pi \circ U^{-1}$ and an action of $GL_k(\R)$ on $\Pi_{d,k}$ given by $V$ such that $\pi \mapsto V \circ \pi$. These actions commute and hence, introduce an action on $GL_k(\R) \times GL_d(\R)$. If $A \subseteq GL_d(\R)$ is a group of invertible linear maps, then $A$ induces an action on measures and hence also an action on distributions, which we keep denoting by $A$. We say that a fractal distribution is \emph{non-singular} with respect to the group $A$, if $a_*P \sim P$ for all $a \in A$.

In this section, we are dealing with product measures and it is therefore convenient to equip the Euclidean space with the maximum metric. In this metric, for measures $\mu$ and $\nu$ on $\R$, we have $(\mu \times \nu)_{(x,y), r} = \mu_{x, r} \times \nu_{y,r}$. For $t, s \geq 0$, we let $(S_t, S_s)$ denote the maps $(x,y) \mapsto (2^t x, 2^s y)$ and $\mu\times\nu \mapsto S_t\mu\times S_s\nu$. Let us next recall some facts from \cite{Hochmanpreprint} on pushing fractal distributions forward under linear or diffeomorphic maps. The first result is \cite[Proposition 1.38]{Hochmanpreprint}.

\begin{proposition} \label{prop-Hochman1.38}
Let $P$ be an ergodic fractal distribution which is non-singular with respect to a group $A\subseteq GL_d(\R)$. Then $E_P(\,\cdot\,)$ is constant on $\overline{A}$-orbits of $\Pi_{d,k}$, where $\overline{A}$ denotes the topological closure of $A$. In particular, if an orbit $O \subseteq \Pi_{d,k}$ of $GL_k(\R) \times \overline{A}$ has non-empty interior, then
$$
E_{P}(\pi) = \min\lbrace k, \dim P \rbrace
$$
for all $\pi \in O$.
\end{proposition}

The next theorem is a consequence of \cite[Theorem 1.10]{HochmanShmerkin2012} and, translated to the language of fractal distributions, it can be found at \cite[Theorem 1.23]{Hochmanpreprint}. 

\begin{theorem} \label{thm-Hochman1.23}
Let $\mu$ be a Radon measure on $\R^d$ which generates an ergodic fractal distribution $P$ along a subsequence. Then
$$
\dim \pi_* \mu \geq E_P(\pi)
$$
for all $\pi \in \Pi_{d,k}$.
\end{theorem}

Before going into the proof of Theorem \ref{prop-dissonance}, we show that for any quasi-Bernoulli measures $\mu$ and $\nu$ on the line, their product $\mu\times\nu$ generates an ergodic fractal distribution almost everywhere, along a subsequence.

\begin{lemma}\label{lemma-productuscselfconformal}
    Let $\Phi$ and $\Psi$ be conformal iterated function systems on $\R$, and let $\mu$ and $\nu$ be quasi-Bernoulli measures associated to $\Phi$ and $\Psi$, respectively. Then there exists a sequence $(n_k)_{k\in\N}$ and an ergodic fractal distribution $P$ with $\dim P \geq \dim (\mu\times\nu)$ such that 
    \begin{equation*}
        \lim_{k\to\infty} \frac{1}{n_k}\int_0^{n_k} \delta_{(\mu\times\nu)_{(x,y),t}}\dd t = P
    \end{equation*}
    for $\mu\times\nu$-almost all $(x,y)$. 
\end{lemma}

\begin{proof}
    Since $(\mu\times\nu)_{(x,y),t} = \mu_{x,t}\times\nu_{y,t}$, an application of Proposition \ref{prop-Hochman1.19} and Lemma \ref{lemma-structureoftangents} shows that there exists an ergodic fractal distribution $P$ with $\dim P \geq \dim(\mu\times\nu)$ and measures $\zeta_1,\zeta_2\in\mathcal{P}(\mathcal{C}^{1+\alpha}(\R))$ such that the measure $(\zeta_1\cdot\mu)_{B(0,1)} \times (\zeta_2\cdot\nu)_{B(0,1)}$ is uniformly scaling and generates the distribution $P$. Arguing similarly as in the proof of Theorem \ref{thm-unif-scaling-2}, we see that for some $(h,g)\in\mathcal{C}^{1+\alpha}(\R)$ there is a subsequence $(n_k)_{k\in\N}$ such that 
    \begin{equation}\label{eq-subsequenceusc}
        \lim_{k\to\infty} \frac{1}{n_k}\int_0^{n_k} \delta_{(h_*\mu \times g_*\nu)_{(h(x),g(y)), t}}\dd t = P
    \end{equation}
    for $\mu\times\nu$-almost all $(x,y)$. On the other hand, by Proposition \ref{prop-Hochman1.9}, we have
    \begin{equation*}
        \lim_{n\to\infty} d_{\rm LP}\biggl( \frac{1}{n}\int_0^n  \delta_{(h_*\mu \times g_*\nu)_{(h(x),g(y)), t}}\dd t, \frac{1}{n}\int_0^n (S_{-\log h'(x)}, S_{-\log g'(y)})_*\delta_{(\mu \times \nu)_{(x,y), t}}\dd t \biggr) = 0
    \end{equation*}
    for $\mu\times\nu$-almost all $(x,y)$. Combining the above with \eqref{eq-subsequenceusc} and $(S_t,S_t)$-invariance of $P$, we see that 
    \begin{equation*}
        \lim_{k\to\infty} \frac{1}{n_k}\int_0^{n_k} \delta_{(\mu \times \nu)_{(x,y), t}}\dd t = (S_{m+\log h'(x)}, S_{m+\log g'(y)})_*P
    \end{equation*}
   for $\mu\times\nu$-almost all $(x,y)$ and for every $m\geq \max\lbrace -\log h'(x), -\log g'(y)\rbrace$. Since linear maps take ergodic fractal distributions to ergodic fractal distributions by \cite[Proposition 1.8]{Hochmanpreprint}, this completes the proof.
\end{proof}

We are now ready to prove Theorem \ref{prop-dissonance}.

\begin{proof}[Proof of Theorem \ref{prop-dissonance}]
Let $P$ be the ergodic fractal distribution given by Lemma \ref{lemma-productuscselfconformal}. Let $\mathtt{i} \in \Gamma^*$ and $j\in\Lambda^*$.   For any open intervals $U$ and $V$ with $\mu(U)>0$ and $\nu(V)>0$, we have
$$
(f_{\mathtt{i}})_*\mu_U \times (g_{\mathtt{j}})_*\nu_V \ll \mu \times \nu.
$$
In particular, by the Lebesgue-Besicovitch differentiation theorem, also $(f_{\mathtt{i}})_*\mu_U \times (g_{\mathtt{j}})_*\nu_V $ generates $P$. On the other hand, since $\mu_U \times \nu_V = ((f_{\mathtt{i}})^{-1}, (g_{\mathtt{j}})^{-1})_*((f_{\mathtt{i}})_*\mu_U \times (g_{\mathtt{j}})_*\nu_V)$, the measure $\mu_U \times \nu_U$ generates at almost every $(x,y)$ the distribution $(S_{\log (f_{\mathtt{i}})'(x)}, S_{\log (g_{\mathtt{j}})'(y)})_*P$ by Proposition \ref{prop-Hochman1.9}. In particular, $P$ is invariant under $(S_{\log (f_{\mathtt{i}})'(x)}, S_{\log (g_{\mathtt{j}})'(y)})$ for all $\mathtt{i} \in \Gamma^*, j\in\Lambda^*$  and $(x,y)\in (f_{\mathtt{i}})(U)\times (g_{\mathtt{j}})(V)$.

Let now $U$ and $V$ be small open intervals centered at the fixed points of $f_{\mathtt{i}}$ and $g_{\mathtt{j}}$, respectively. Taking $\diam(U), \diam(V)\to 0$ and using the $(S_t,S_t)$-invariance of $P$ and continuity of $(t,s)\mapsto (S_t,S_s)_*P$, we see that $P$ is invariant under $(S_{\log \lambda(f_{\mathtt{i}}) -\log \lambda(g_{\mathtt{j}})}, {\rm Id})$ for all $\mathtt{i} \in \Gamma^*, j\in\Lambda^*$ .  Write $D= \lbrace\log \lambda(f_{\mathtt{i}}) -\log \lambda(g_{\mathtt{j}}) :\ \mathtt{i} \in \Gamma^*, j\in\Lambda^*\rbrace$.  Recall our assumption that for any $\varepsilon>0$, there exist $\mathtt{i} \in \Gamma^*, j\in\Lambda^*$ such that $0<|\log \lambda(f_{\mathtt{i}}) -\log \lambda(g_{\mathtt{j}})|<\varepsilon$.  Using this, and the relation $\lambda(f\circ f)=\lambda(f)\lambda(f)$, we deduce that  the set $D$ is dense either in $[0,+\infty)$ or in $(-\infty,0]$.  Without loss of generality, let us suppose that  $D$ is dense in $[0,+\infty)$.  Let $D'=D\cap [0,+\infty)$ and $A = \lbrace (S_{t}, {\rm Id}) : t\in D' \rbrace$.   Write $\pi$ for the map $(x,y) \mapsto x+y$. Since the orbit of $\pi$ under $GL_1(\R) \times \overline{A}$ is $\lbrace (x,y) \mapsto \alpha(2^{t} x + y) : \alpha> 0 \text{ and }  t\ge 0 \rbrace$ which has nonempty interior and contains $\pi$, Proposition \ref{prop-Hochman1.38} asserts that $E_P(\pi) = \min \lbrace 1, \dim P \rbrace$. Since $\dim P \geq \dim (\mu \times \nu) = \dim \mu + \dim \nu$ by the choice of $P$ and exact-dimensionality of $\mu$ and $\nu$, and $\dim(\mu*\nu) = \dim \pi_*\mu \geq E_P(\pi)$ by Theorem \ref{thm-Hochman1.23}, we have completed the proof. 
\end{proof}

\section{Self-affine measures} \label{sec-self-affine-proj}

In this section, we prove Theorem \ref{self-affine_projection}. Recall that if $\Phi = (f_i)_{i \in \Lambda}$ is a diagonal iterated function system, then the maps $f_i\colon \R^2 \to \R^2$, $f_i(x)=A_ix+a_i$, are affine such that
\begin{equation*}
  A_i =
  \begin{pmatrix}
    \rho_i & 0 \\ 
    0 & \lambda_i
  \end{pmatrix}
\end{equation*}
for all $i \in \Lambda$.

\begin{proposition}\label{thm-self-affine-usc}
  Let $\Phi$ be a diagonal iterated function system on $\R^2$ with attractor $K$ satisfying the rectangular strong separation condition. Then any self-affine measure $\mu$ on $K$ is uniformly scaling and generates an ergodic fractal distribution. Furthermore, if $\mu$ has simple Lyapunov spectrum, then $P = (S_{-\log \rho_\iii}, S_{-\log \lambda_\mathtt{j}})_* P$ for all $\iii, \mathtt{j}\in\Lambda^*$. 
\end{proposition}

Relying on Proposition \ref{thm-self-affine-usc}, the proof of Theorem \ref{self-affine_projection} is essentially identical to the last paragraph of the proof of Theorem \ref{prop-dissonance}.

\begin{proof}[Proof of Theorem \ref{self-affine_projection}]
    Let $\mu$ be a self-affine measure on the attractor of $\Phi$ with simple Lyapunov spectrum, and let $P$ be the ergodic fractal distribution generated by $\mu$ given by Proposition \ref{thm-self-affine-usc}. Let $\pi\in \Pi_{2,1}$ be a non-principal orthogonal projection and write $\pi(x,y) = ax + by$ where we assume $a,b >0$; the case $a,b < 0$ is identical. Choose $N$ large enough so that $2^{-N} \leq a,b\leq 2^N$.  Recall our  assumption that for any $\varepsilon>0$, there exist $\mathtt{i}, \mathtt{j}\in \Lambda^*$ such that  $0<|\log \rho_{\mathtt{i}} - \log \lambda_{\mathtt{j}}|<\varepsilon$.  It follows that  set $D = \lbrace -\log \rho_\iii + \log \lambda_{\mathtt{j}} : \iii,\mathtt{j}\in\Lambda^*\rbrace$ is dense either  in $[0,N]$ or in $[-N,0]$.  Without loss of generality, we suppose that $D$ is dense in $[0,N]$. Let $A = \lbrace (S_t, {\rm Id}): t\in D\cap [0,N]\rbrace$. Since $\mu$ has simple Lyapunov spectrum, it follows from Proposition \ref{thm-self-affine-usc} that $P$ is invariant under every map in $A$. On the other hand, the orbit of $\pi$ under $GL_1(\R)\times \overline{A}$ contains $\pi$ and has non-empty interior, whence $E_P(\pi) = \min \lbrace 1, \dim P \rbrace = \min \lbrace 1, \dim \mu \rbrace$ by Propositions \ref{prop-Hochman1.38} and \ref{prop-Hochman1.19}. The proof is now finished as, by Theorem \ref{thm-Hochman1.23}, we have $\dim \pi_*\mu \geq E_P(\pi)$.
\end{proof}

The rest of the section is devoted to the proof of Proposition \ref{thm-self-affine-usc}. By \cite[\S 4.3]{Hochmanpreprint}, the first statement is known when the Lyapunov exponents of $\mu$ coincide and we may thus assume that $\mu$ has simple Lyapunov spectrum throughout the proof. Without loss of generality, we assume $y$-axis to be the major asymptotic contracting direction. It follows that for $\bar{\mu}$-almost every $\mathtt{i}\in\Lambda^\N$ and every large enough $n$ (depending on $\mathtt{i}$), we have $\lambda_{\mathtt{i}|_n} < \rho_{\mathtt{i}|_n}$. Recall that $\bar{\mu}$ is the Bernoulli measure whose canonical projection is $\mu$.

Let $R = [a,b] \times [c,d]$ be a rectangle. Recall that, if $\mu$ is a measure such that $0<\mu(R)<\infty$, then $\mu_R = \mu(R)^{-1}\mu|_R$. With a slight abuse of notation, let $\mu^R = H_*\mu_R$, where $H(x,y) = (\frac{2x-(b+a)}{b-a},\frac{2y-(c+d)}{d-c})$ is the linear map translating and rescaling $R$ onto $[-1,1]^2$. From now on, let $\pi \in \Pi_{2,1}$ be the orthogonal projection onto the $x$-axis. Let $\mu = \int \delta_x \times \mu_x\dd\pi_*\mu(x)$ be the disintegration of $\mu$ with respect to $\pi$, where $\mu_x$ is a measure on $\R$ for each $x$. We also write $\mu_\mathtt{i} = \mu_{\pi(\Pi(\mathtt{i}))}$ for all $\mathtt{i}\in\Lambda^\N$.

\begin{lemma}\label{lemma-productstructure}
    Let $\mu$ be a Borel probability measure on $\R^2$. Then, for $\pi_*\mu$-almost every $x$, we have
    $$
      d_{\rm LP}(\mu^{\pi^{-1}(B(x,2^{-r}))}, (\pi_*\mu)_{x,r}\times \mu_x) \to 0
    $$
    as $r\to \infty$. 
\end{lemma}

\begin{proof}
    By Lusin's theorem, for every $n\in\N$ there exists a set $E_n\subseteq \R$ with $\pi_*\mu(E_n)\geq 1-1/n$ such that $x\mapsto \mu_x$ is continuous on $E_n$. On the other hand, by the Lebesgue-Besicovitch density point theorem, there is a set $E_n'\subseteq E_n$ such that $\pi_*\mu(E_n') = \pi_*\mu(E_n)$ and
    $$
    \frac{\pi_*\mu(E_n' \cap B(x,2^{-r}))}{\pi_*\mu(B(x,2^{-r}))} \to 1
    $$
    as $r\to \infty$ for all $x\in E_n'$. Fix $n \in \N$ and $x\in E_n'$. If $H_r$ denotes the affine map which translates and rescales $\pi^{-1}(B(x,2^{-r})) \cap \R \times [-1,1]$ onto $[-1,1]^2$, we have
    \begin{align*}
    d_{\rm LP}(\mu^{\pi^{-1}(B(x,2^{-r}))}, (\pi_*\mu)_{x,r}\times \mu_x)
    &=d_{\rm LP}\biggl( \int_{B(x,2^{-r})} H_r(\delta_y\times\mu_y)\dd(\pi_*\mu)_{x,r}(y), \\
    &\qquad\qquad\int_{B(x,2^{-r})}H_r(\delta_y\times\mu_x) \dd(\pi_*\mu)_{x,r}(y)\biggr) \to 0 
\end{align*}
    as $r\to \infty$, since $d_{\rm LP}(\mu_y, \mu_x) = o(r)$ for $y\in E_n' \cap B(x,2^{-r})$. Since $\pi_*\mu(\bigcup_{n=1}^\infty E_n') = 1$, this concludes the proof.
\end{proof}

Combining Lemma \ref{lemma-productstructure} with the self-affinity of $\mu$, we deduce that outside a negligible set of scales, magnifications of $\mu$ also have a product-structure. 

\begin{lemma}\label{lemma-scenerymeasures}
  Let $\Phi$ be a diagonal iterated function system on $\R^2$ with attractor $K$ satisfying the rectangular strong separation condition and $\mu$ be a self-affine measure on $K$. Then there exists $c>0$ such that, for $\bar{\mu}$-almost every $\mathtt{i}\in\Lambda^\N$, we have
    $$
    \frac{1}{n}\sum_{k=0}^{n-1}d_{\rm LP}(\mu_{\Pi(\mathtt{i}), -\log\lambda_{\mathtt{i}|_k} + c}, S_c T_{\Pi(\sigma^k \mathtt{i})} ((\pi_*\mu)_{\pi(\Pi(\sigma^k \mathtt{i})), \log \rho_{\mathtt{i}|_k} - \log \lambda_{\mathtt{i}|_k}} \times \mu_{\sigma^k \mathtt{i}})) \to 0
    $$
    as $n\to\infty$.
\end{lemma}

\begin{proof}
    Let $c = \max \lbrace -\log \dist(\varphi_i([-1,1]^2), \varphi_j([-1,1]^2):\ i,j \in \Lambda\text{ such that }i\neq j\rbrace$ and recall that, for $\bar{\mu}$-almost every $\mathtt{i}$ and every large enough $n$, we have $\lambda_{\mathtt{i}|_n} < \rho_{\mathtt{i}|_n}$. By the choice of $c$, the ball $B(\Pi(\mathtt{i}), 2^{\log \lambda_{\mathtt{i}|_n} - c})$ only intersects $\varphi_{\mathtt{i}|_n}([-1,1]^2)$. Writing
    $$
      R_{\mathtt{i}, n} = \varphi_{\mathtt{i}|_n}^{-1} B(\Pi(\mathtt{i}), 2^{\log \lambda_{\mathtt{i}|_n} - c}),
    $$
    we see that $R_{\mathtt{i},n}$ is the rectangle $[-2^{\log \lambda_{\mathtt{i}|_n} - \log \rho_{\mathtt{i}|_n}-c}, 2^{\log \lambda_{\mathtt{i}|_n} - \log \rho_{\mathtt{i}|_n}-c}] \times [-2^{-c},2^{-c}]$ translated so that its centre is at $\Pi(\sigma^n \mathtt{i})$. Hence, by the self-affinity,
    $$
	\mu_{\Pi(\mathtt{i}), -\log\lambda_{\mathtt{i}|_n} + c} = \mu^{R_{\mathtt{i}, n}}.
    $$
    Let $\varepsilon>0$ and notice that, by Egorov's theorem and Lemma \ref{lemma-productstructure}, there exists a set $A\subset\Lambda^\N$ and $N\geq1$ such that $\bar{\mu}(A)>1-\varepsilon$ and
    $$
    d_{\rm LP}(\mu^{\pi^{-1}(B(\Pi(\iii),2^{-r}))},\ (\pi_*\mu)_{\Pi(\iii),r}\times \mu_\iii)<\varepsilon
    $$
    for all $\iii\in A$ and $r\geq N$. Thus, 
    $$
    d_{\rm LP}(\mu^{R_{\mathtt{i}, k}}, S_c (T_{\Pi(\sigma^k \mathtt{i})})_* ((\pi_*\mu)_{\pi(\Pi(\sigma^k \mathtt{i})), \log \rho_{\mathtt{i}|_k} - \log \lambda_{\mathtt{i}|_k}} \times \mu_{\sigma^k \mathtt{i}}))<\varepsilon
    $$
    for every $k$ sufficiently large such that $\sigma^k\iii\in A$. Therefore,
    \begin{align*}
      \frac{1}{n}\sum_{k=0}^{n-1}d_{\rm LP}&(\mu^{R_{\mathtt{i}, k}}, S_c (T_{\Pi(\sigma^k \mathtt{i})})_* ((\pi_*\mu)_{\pi(\Pi(\sigma^k \mathtt{i})), \log \rho_{\mathtt{i}|_k} - \log \lambda_{\mathtt{i}|_k}} \times \mu_{\sigma^k \mathtt{i}}))\\
      &\leq\frac{1}{n}\sum_{k=0}^{n-1}(\mathds{1}\{\sigma^k\iii\in A\}\varepsilon+\mathds{1}\{\sigma^k\iii\notin A\})+\frac{2N}{n}\\
      &\leq\varepsilon+\frac{1}{n}\sum_{k=0}^{n-1}\mathds{1}\{\sigma^k\iii\notin A\}+\frac{2N}{n} \to \eps+\bar{\mu}(\Lambda^\N\setminus A) \le 2\varepsilon
    \end{align*}
  as $n \to \infty$. Since $\varepsilon>0$ was arbitrary, the claim follows.
\end{proof}

We prove Proposition \ref{thm-self-affine-usc} by first showing that the product measures $\pi_*\mu \times\mu_\mathtt{i}$ are uniformly scaling along a subsequence for $\bar{\mu}$-almost every $\mathtt{i}\in\Lambda^\N$, and then apply Lemma \ref{lemma-scenerymeasures} to transfer the property to the measure $\mu$.

\begin{lemma}\label{lemma-productusc}
    Let $\Phi$ be a diagonal iterated function system on $\R^2$ with attractor $K$ satisfying the rectangular strong separation condition and $\mu$ be a self-affine measure on $K$. Then there exists an ergodic fractal distribution $P$ with $\dim P \geq \dim \mu$ and a sequence $(n_k)_{k\in\N}$ such that
    \begin{equation}\label{eq-lemma-productusc}
        \lim_{k\to\infty} \frac{1}{n_k}\int_0^{n_k}\delta_{(\pi_*\mu\times\mu_\iii)_{(x,y),t}}\dd t = P
    \end{equation}
    for $\bar{\mu}$-almost all $\mathtt{i}\in\Lambda^\N$ and $\pi_*\mu\times\mu_\iii$-almost all $(x,y)$. Furthermore,
    \begin{equation*}
      P = (S_{-\log \rho_\iii}, S_{-\log \lambda_{\mathtt{j}}})_*P
    \end{equation*}
    for all $\iii,\mathtt{j}\in\Lambda^*$.
\end{lemma}

\begin{proof}
    Let $\pi' \in \Pi_{2,1}$ be the orthogonal projection onto the $y$-axis. Standard arguments (see also \cite[Theorem 4.2]{Kempton2015} for a more general result) using the relation 
    \begin{equation}\label{eq-almostselfsimilar}
    (\mu_\iii)_{\pi'(\Pi(\iii)), t} =(\mu_{\sigma(\iii)})_{\pi'(\Pi(\sigma(\iii))), t+\log \lambda_{\iii|_1}}
    \end{equation}
    for all large enough $t$ show that, for $\bar{\mu}$-almost every $\iii$, the measure $\mu_\iii$ is uniformly scaling and generates the ergodic fractal distribution
    \begin{equation*}
        \int_{\Lambda^\N} \int_0^{-\log \lambda_{\iii|_1}} S_t (T_{\pi'(\Pi(\iii))})_*\mu_\iii \dd t\dd \bar{\mu}(\iii).
    \end{equation*}
    Since $\pi_*\mu$ is a self-similar measure on $\R$ and $\dim\pi_*\mu\times \mu_\iii = \dim \mu$ for $\bar{\mu}$-almost every $\iii$ by \cite[Theorem 2.7]{Barany2015} (see also \cite{BaranyKaenmaki2017}), the rest of the proof of \eqref{eq-lemma-productusc} is analogous to the proof of Proposition \ref{lemma-productuscselfconformal}. The final claim is easy to see by using \eqref{eq-almostselfsimilar}, self-similarity of $\pi_*\mu$, and the Lebesgue-Besicovitch density point theorem. 
\end{proof}

We are now ready to prove Proposition \ref{thm-self-affine-usc}. Since the ideas in the proof have already appeared in the proof of Theorem \ref{thm-unif-scaling-2}, we will only provide a sketch.

\begin{proof}[Proof of Proposition \ref{thm-self-affine-usc}]
    Let $Q$ be an ergodic component of a tangent distribution of $\mu$. Combining Lemma \ref{lemma-productstructure} with a minor modification of Lemma \ref{lemma-structureoftangents}, we see that outside a set of zero $\mu$-measure, $Q$-almost every measure is uniformly scaling, generates $Q$ and is of the form
    \begin{equation*}
        ((\nu\cdot\pi_*\mu) \times \mu_\iii)_{(x,y),t}
    \end{equation*}
    for some $\iii\in\Lambda^\N$, $(x,y)\in\R^2$, $t\geq 0$, and $\nu$ supported on the space of linear maps $h\colon\R\to\R$, $h(y) = r_h y + a_h$. It follows from Lemma \ref{lemma-productstructure} that $r_h$ is in the closure of $\lbrace \rho_\iii \lambda_\mathtt{j}^{-1} : \iii, \mathtt{j}\in\Lambda^*\rbrace$. Similarly as in the proof of Theorem \ref{thm-unif-scaling-2}, Propositions \ref{prop-Hochman1.9} and \ref{prop-main-1} together assert that
    \begin{equation*} 
        \lim_{n\to\infty} \iint d_{\rm LP}\biggl(Q, \frac{1}{n}\int_0^n  (S_{-\log r_h}, {\rm Id})_* \delta_{(\pi_*\mu\times\mu_\iii)_{(x,y),t}}\dd t\biggr)\dd \pi_*\mu\times\mu_\iii(x,y)\dd\nu(h) = 0  
    \end{equation*}
    On the other hand, if $P$ is the fractal distribution of Lemma \ref{lemma-productusc}, combining the above with the dominated convergence theorem and applying the triangle inequality, we see that 
    \begin{equation*}
        Q = (S_{-\log r_h},{\rm Id})_* P = P,
    \end{equation*}
    where the last equality follows since $r_h \in \overline{\lbrace \rho_\iii \lambda_\mathtt{j}^{-1}: \iii, \mathtt{j}\in\Lambda^*\rbrace}$.
\end{proof}

\begin{ack}
  We thank Amir Algom for useful comments.
\end{ack}


\end{document}